\documentclass[11pt]{amsart}

\usepackage[english]{babel}
\usepackage{ucs}
\usepackage[T1]{fontenc}
\usepackage[utf8x]{inputenc}

\usepackage{url}
\usepackage{dsfont}
\usepackage{enumitem}

\usepackage{times}
\usepackage[mathscr]{euscript}

\usepackage[colorlinks=true]{hyperref}

\newtheorem{defi}{Definition}[section]
\newtheorem{lem}[defi]{Lemma}

\newtheorem{prop}[defi]{Proposition}
\newtheorem{theo}[defi]{Theorem}
\newtheorem{cor}[defi]{Corollary}
\newtheorem{rk}[defi]{Remark}
\newtheorem{ex}[defi]{Example}

\newcommand{\dd}{\mathrm{d}}
\newcommand{\R}{\mathbb{R}}

\newcommand{\Exp}{\mathbf{E}}
\renewcommand{\Pr}{\mathbf{P}}
\newcommand{\ind}[1]{\mathds{1}_{\{#1\}}}

\newcommand{\dto}{\downarrow}

\renewcommand{\bar}{\overline}
\renewcommand{\hat}{\widehat}
\renewcommand{\tilde}{\widetilde}

\newcommand{\Ws}{\mathrm{W}}

\newcommand{\Ps}{\mathcal{P}}
\newcommand{\tPs}{\tilde{\Ps}}
\newcommand{\bPs}{\bar{\Ps}}
\newcommand{\Pst}{\Ps_*}

\newcommand{\PP}{\mathbb{P}}
\newcommand{\QQ}{\mathbb{Q}}

\newcommand{\FreEne}{\mathscr{F}}

\newcommand{\GenFun}{\mathscr{G}}
\newcommand{\IntEne}{\mathscr{W}}
\newcommand{\intene}{W}
\newcommand{\PinEne}{\mathscr{V}}

\newcommand{\Ent}{\mathscr{S}}
\newcommand{\RelEnt}{\mathscr{R}}
\newcommand{\Rate}{\mathscr{I}}
\newcommand{\tRate}{\tilde{\Rate}}
\newcommand{\bRate}{\bar{\Rate}}
\newcommand{\bFreEne}{\bar{\FreEne}}
\newcommand{\bEnt}{\bar{\Ent}}
\newcommand{\bGenFun}{\bar{\GenFun}}
\newcommand{\bIntEne}{\bar{\IntEne}}

\newcommand{\tX}{\tilde{X}}
\newcommand{\tx}{\tilde{x}}

\newcommand{\tbx}{\tilde{\mathbf{x}}}
\newcommand{\tby}{\tilde{\mathbf{y}}}

\newcommand{\bx}{\mathbf{x}}

\newcommand{\Mdn}{M_{d,n}}
\newcommand{\Rdn}{(\R^d)^n}

\newcommand{\tP}{\tilde{P}}
\newcommand{\tp}{\tilde{p}}

\newcommand{\tPP}{\tilde{\PP}}

\newcommand{\bPP}{\bar{\PP}}

\newcommand{\tZ}{\tilde{Z}}
\newcommand{\tz}{\tilde{z}}

\newcommand{\tmu}{\tilde{\mu}}
\newcommand{\bmu}{\bar{\mu}}
\newcommand{\bnu}{\bar{\nu}}
\newcommand{\blambda}{\bar{\lambda}}

\newcommand{\Tc}{\mathrm{T}} 
\newcommand{\tc}{\mathrm{t}} 

\newcommand{\dP}{d_{\mathrm{P}}}
\newcommand{\bdP}{\bar{d}_{\mathrm{P}}}

\newcommand{\kB}{\mathrm{k}}

\newcommand{\MV}{\textsc{mv}}
\newcommand{\RB}{\textsc{rb}}

\title{Equilibrium large deviations for mean-field systems with translation invariance}
\author{Julien Reygner}
\address{Universit\'e Paris-Est, CERMICS (ENPC), F-77455 Marne-la-Vall\'ee}
\email{julien.reygner@enpc.fr}
\thanks{This work is partially supported by: the European Research Council under the European Union's Seventh Framework Programme (FP/2007-2013) / ERC Grant Agreement number 614492; the Chaire Risques Financiers, Fondation du Risque; and the French National Research Agency (ANR) under the programs ANR-12-BLAN Stab and ANR-17-CE40-0030 EFI}


\makeatletter

\def\section{\@startsection{section}{1}%
\z@{.7\linespacing\@plus\linespacing}{.5\linespacing}%
{\normalfont\bfseries\centering}}

\def\@settitle{\begin{center}%
  \baselineskip14\p@\relax
    \bfseries
    \LARGE\@title
  \end{center}%
}

\def\@setauthors{%
  \begingroup
  \trivlist
  \centering\footnotesize \@topsep30\p@\relax
  \advance\@topsep by -\baselineskip
  \item\relax
  \andify\authors
  \def\\{\protect\linebreak}%
 {\Large\authors}%
  \endtrivlist
  \endgroup
}

\def\maketitle{\par
  \@topnum\z@ 
  \@setcopyright
  \thispagestyle{firstpage}
  \ifx\@empty\shortauthors \let\shortauthors\shorttitle
  \else \andify\shortauthors
  \fi
  \@maketitle@hook
  \begingroup
  \@maketitle
  \toks@\@xp{\shortauthors}\@temptokena\@xp{\shorttitle}%
  \toks4{\def\\{ \ignorespaces}}
  \edef\@tempa{%
    \@nx\markboth{\the\toks4
      \@nx{\the\toks@}}{\the\@temptokena}}%
  \@tempa
  \endgroup
  \c@footnote\z@
  \def\do##1{\let##1\relax}%
  \do\maketitle \do\@maketitle \do\title \do\@xtitle \do\@title
  \do\author \do\@xauthor \do\address \do\@xaddress
  \do\email \do\@xemail \do\curraddr \do\@xcurraddr
  \do\commby \do\@commby
  \do\dedicatory \do\@dedicatory \do\thanks \do\thankses
  \do\keywords \do\@keywords \do\subjclass \do\@subjclass
}
\makeatother

\begin{document}

\begin{abstract}
  We consider particle systems with mean-field interactions whose distribution is invariant by translations. Under the assumption that the system seen from its centre of mass be reversible with respect to a Gibbs measure, we establish large deviation principles for its empirical measure at equilibrium. Our study covers the cases of McKean-Vlasov particle systems without external potential, and systems of rank-based interacting diffusions. Depending on the strength of the interaction, the large deviation principles are stated in the space of centered probability measures endowed with the Wasserstein topology of appropriate order, or in the orbit space of the action of translations on probability measures. An application to the study of atypical capital distribution is detailed.
\end{abstract}

\maketitle


\section{Introduction}

This work is dedicated to the study of the large deviations of the empirical measure of particle systems at equilibrium exhibiting the following formal features:
\begin{enumerate}[label=(\alph*),ref=\alph*]
  \item\label{cond:a} they are reversible with respect to an explicit Gibbs measure;
  \item\label{cond:b} the particles are coupled through mean-field interactions;
  \item\label{cond:c} their distribution is invariant under spatial translations.
\end{enumerate}
The typical models that we aim to study include McKean-Vlasov particle systems without external potential, whose mean-field limit allows to approximate the granular media equation~\cite{BenRoyTalVal98, BenRoyVal98, Mal03, CatGuiMal08}, and systems of one-dimensional diffusions interacting through their rank, which arise in the probabilistic interpretation of scalar nonlinear conservation laws~\cite{BosTal96, BosTal97, Jou00:MCAP, Shk12, JouRey13}. Both models also appear in mathematical finance, in the modelling of inter-bank borrowing and lending~\cite{FouSun13} and of stable equity markets~\cite{Fer02, JouRey15}, respectively.

For McKean-Vlasov particle systems \emph{with} an external potential, which in general satisfy the conditions~\eqref{cond:a} and~\eqref{cond:b} but not~\eqref{cond:c}, the large deviations of the empirical measure of the particle system under its equilibrium measure are governed by the \emph{free energy} functional, which combines entropic and energetic contributions. Prefiguring the interpretation by Otto~\cite{JorKinOtt98,Ott01} and Carrillo, McCann and Villani~\cite{CarMcCVil03,CarMcCVil06} of (nonlinear) Fokker-Planck equations as functional gradient flows, Dawson and Gärtner~\cite{DawGar86,DawGar87,DawGar89} showed that, for such systems, the free energy plays the role of a quasipotential, in the sense of the Freidlin-Wentzell theory. Thus it not only describes the \emph{static} large scale properties of the particle system, such as typical configurations or possible phase transitions, but it also sheds light on its large scale \emph{dynamics}, providing both typical paths and fluctuation rates.

For models satisfying the condition~\eqref{cond:c}, translation invariance generally prevents ergodicity, so that there is no equilibrium measure for the original particle system. Still it was noted in~\cite{Mal03} for McKean-Vlasov systems, and in~\cite{JouMal08,PalPit08} for rank-based interacting diffusions, that under suitable assumptions on the interactions between the particles, a stationary behaviour can be observed for the particle system \emph{seen from its centre of mass}. Centering the particle system induces a conserved quantity in its evolution, and the purpose of this article is to understand the effect of this constraint on its equilibrium large deviations. To the best of the author's knowledge, this is the first study in this direction.

For such systems, a free energy functional can still be defined, with an energetic contribution depending only on the interaction between the particles. Thus it may be expected that, under the assumption that the centered particle system be ergodic, the large deviations of its empirical measure at equilibrium be described by this free energy functional, restricted to the space of \emph{centered} probability measures. The first result of this article, Theorem~\ref{theo:Wass}, provides a rigorous formulation of this assertion; however, it only holds under the assumption that the interaction between particles be strong enough, in a sense to be made precise below --- typically, for McKean-Vlasov systems with an interaction potential growing faster than linearly. In contrast, when this assumption is not satisfied, which turns out to be the case for systems of rank-based interacting diffusions, we show that the rate function may fail to have compact level sets, so that the expected large deviation principle \emph{does not hold}. This is formally explained by the following two facts: the topology on which a large deviation principle can be expected to hold depends on the strength of the interaction; and on too weak topologies, the space of centered probability measures is not closed.

In order to connect the free energy functional to the equilibrium large deviations of the particle system without restriction on the strength of the interaction, and thereby cover the case of rank-based interacting diffusions, we avoid resorting to the notion of centered probability measures, and rather work at the level of the \emph{orbit} of the empirical measure of the particle system at equilibrium, under the action of translations. This provides an equivalent description of the particle system, however the quotient topology on the orbit space becomes weak enough for a large deviation principle to hold without any assumption on the strength of the interaction. This is the second main result of the article, Theorem~\ref{theo:quot}, which is weaker than the first in the sense that it is implied by the latter, but holds under less restrictive assumptions.

The adaptation of these results to the specific examples of McKean-Vlasov particle systems and systems of rank-based interacting diffusions are stated as corollaries. For the latter example, the large deviation principle allows to associate a notion of free energy to scalar nonlinear conservation laws, which complements, at the level of the stationary measure, the results by Dembo, Shkolnikov, Varadhan and Zeitouni on finite time intervals~\cite{DemShkVarZei16}. As an application, we discuss at the end of the article the estimation of the probability of an atypical capital distribution in the framework of Fernholz' \emph{Stochastic Portfolio Theory}~\cite{Fer02}.

\subsection*{Outline of the article} The notations and main results of the article are presented in Section~\ref{s:main}. The proof of our two main theorems is based on the approximation of the particle system without external potential by a particle system with a small external potential. The large deviation results for this approximating system are presented in Section~\ref{s:conf}, and the control of these results when the external potential vanishes is studied in Section~\ref{s:pfmain}. The application of the main results to the particular cases of McKean-Vlasov particle systems, and systems of rank-based interacting diffusions, is detailed in Section~\ref{s:appl}. A technical result on the metrisability of the quotient topology is proved in Appendix~\ref{s:app}. 


\section{Notations and main results}\label{s:main}
\subsection{Spaces of probability measures} For $d \geq 1$, we denote by $\Ps(\R^d)$ the space of Borel probability measures on $\R^d$. It is endowed with the topology of weak convergence~\cite[Chapter~1, p.~7]{Bil99}, which makes it a Polish space~\cite[Theorem~6.8, p.~73]{Bil99}.

For all $y \in \R^d$, we define the \emph{translation by $y$} as the operator $\tau_y : \Ps(\R^d) \to \Ps(\R^d)$ such that, for all $\mu \in \Ps(\R^d)$,
\begin{equation*}
  \int_{x \in \R^d} f(x) \dd\tau_y\mu(x) = \int_{x \in \R^d} f(x+y) \dd\mu(x),
\end{equation*}
for all measurable and bounded functions $f : \R^d \to \R$. It is known that the operator $\tau_y$ is continuous on $\Ps(\R^d)$.

For all $p \geq 1$, we denote by $\Ps_p(\R^d)$ the space of Borel probability measures on $\R^d$ with a finite $p$-th order moment. It is endowed with the \emph{Wasserstein topology of order $p$}~\cite[Definition~6.8, p.~96]{Vil09}, which makes it a Polish space~\cite[Theorem~6.18, p.~104]{Vil09}.

The Wasserstein topology is stronger than the topology induced on $\Ps_p(\R^d)$ by the topology of weak convergence on $\Ps(\R^d)$, so that for any $y \in \R^d$, the translation $\tau_y$ is continuous on $\Ps_p(\R^d)$.

We denote by $\tPs_p(\R^d)$ the subset of \emph{centered} probability measures with a finite $p$-th order moment, and define the \emph{centering operator} $\Tc : \Ps_p(\R^d) \to \tPs_p(\R^d)$ by
\begin{equation*}
  \Tc\mu = \tau_{-\xi}\mu, \qquad \xi := \int_{x \in \R^d} x \dd\mu(x),
\end{equation*}
for all $\mu \in \Ps_p(\R^d)$. It is easily checked that $\Tc$ is continuous on $\Ps_p(\R^d)$, and that $\tPs_p(\R^d)$ is a closed subset of $\Ps_p(\R^d)$, hence it is a Polish space itself.

In the sequel of this article, we shall consider probability measures defined on the respective Borel $\sigma$-fields of the topological spaces $\Ps(\R^d)$, $\Ps_p(\R^d)$ and $\tPs_p(\R^d)$.

\subsection{Energy functional and Gibbs measure} Throughout the article, the \emph{temperature parameter} $\sigma^2>0$ is fixed.

The physical systems which we aim to study are described by an \emph{energy functional}
\begin{equation*}
  \IntEne : \Ps(\R^d) \to [0,+\infty]
\end{equation*}
satisfying the following set of conditions:
\begin{enumerate}[label=(TI),ref=TI]
  \item\label{ass:TI} translation invariance: for all $y \in \R^d$, for all $\mu \in \Ps(\R^d)$, $\IntEne[\tau_y\mu] = \IntEne[\mu]$;
\end{enumerate}\begin{enumerate}[label=($\sigma$F),ref=$\sigma$F]
  \item\label{ass:F} $\sigma$-finiteness: if $\mu$ has compact support, then $\IntEne[\mu] < +\infty$;
\end{enumerate}\begin{enumerate}[label=(LSC),ref=LSC]
  \item\label{ass:LSC} the function $\IntEne$ is lower semicontinuous on $\Ps(\R^d)$;
\end{enumerate}\begin{enumerate}[label=(GC),ref=GC]
  \item\label{ass:GC} growth control: there exists $\ell \geq 1$ and $\kappa_\ell > 0$ such that $\IntEne[\mu]=+\infty$ if $\mu \not\in \Ps_\ell(\R^d)$, and
  \begin{equation*}
    \forall \tmu \in \tPs_\ell(\R^d), \qquad \IntEne[\tmu] \geq \kappa_\ell \int_{x \in \R^d} |x|^\ell \dd\tmu(x).
  \end{equation*}
\end{enumerate}

For all $n \geq 2$, the \emph{energy} of a configuration $\bx = (x_1, \ldots, x_n) \in \Rdn$ of a system with $n$ particles is defined by
\begin{equation}\label{eq:intene}
  \intene_n(\bx) := \IntEne[\pi_n(\bx)],
\end{equation}
where
\begin{equation}\label{eq:pin}
  \pi_n(\bx) := \frac{1}{n}\sum_{i=1}^n \delta_{x_i} \in \Ps(\R^d)
\end{equation}
is the empirical measure of the configuration $\bx$. Notice that Assumption~\eqref{ass:F} ensures that $\intene_n(\bx) < +\infty$ for any configuration $\bx \in \Rdn$. 

The Gibbs density $\exp(-\frac{2n}{\sigma^2} \intene_n(\bx))$ naturally associated with the energy function $\intene_n$ is never integrable on $\Rdn$, because Assumption~\eqref{ass:TI} implies that $\intene_n(\bx)$ is invariant under the translations $(x_1, \ldots, x_n) \mapsto (x_1+\zeta, \ldots, x_n+\zeta)$, $\zeta \in \R^d$. However, introducing the linear subspace 
\begin{equation}\label{eq:Mdn}
  \Mdn := \{\tbx = (\tx_1, \ldots, \tx_n) \in \Rdn : \tx_1 + \cdots + \tx_n = 0\},
\end{equation}
and denoting by $\dd\tbx$ the Lebesgue measure on $\Mdn$, we get the following first result.

\begin{lem}[Finiteness of the partition function]\label{lem:tP}
  Let $\IntEne : \Ps(\R^d) \to [0,+\infty]$ be an energy functional satisfying Assumptions~\eqref{ass:TI}, \eqref{ass:F}, \eqref{ass:LSC} and~\eqref{ass:GC}. For all $n \geq 2$, we have
  \begin{equation}\label{eq:tZ}
    \tZ_n := \int_{\tbx\in\Mdn} \exp\left(-\frac{2n}{\sigma^2} \intene_n(\tbx)\right) \dd\tbx \in (0,+\infty),
  \end{equation}
  where the function $\intene_n$ is defined by~\eqref{eq:intene}.
\end{lem}
\begin{proof}
  The combination of Assumptions~\eqref{ass:F} and~\eqref{ass:LSC} ensures that the function $\tbx \mapsto \exp(-\frac{2n}{\sigma^2} \intene_n(\tbx))$ is positive and measurable, so that $\tZ_n$ is well-defined as an element of $(0,+\infty]$. Using~\eqref{eq:Mdn}, Assumption~\eqref{ass:GC} and the trivial bound
  \begin{equation*}
    \sum_{i=1}^n |\tx_i|^\ell \geq \sum_{i=1}^{n-1} |\tx_i|^\ell,
  \end{equation*}
  we get the inequality
  \begin{equation*}
    \tZ_n \leq \int_{\tbx\in\Mdn} \exp\left(-\frac{2\kappa_\ell}{\sigma^2} \sum_{i=1}^{n-1}|\tx_i|^\ell\right) \dd\tbx,
  \end{equation*}
  whose right-hand side is proven to be finite by using the parametrisation of $\tbx = (\tx_1, \ldots, \tx_n) \in \Mdn$ by $(\tx_1, \ldots, \tx_{n-1}) \in (\R^d)^{n-1}$.
\end{proof}

\begin{defi}[Gibbs measure]\label{defi:tP}
  Under the assumptions of Lemma~\ref{lem:tP}, we denote by $\tP_n$ the probability measure on $\Rdn$ with density
  \begin{equation}\label{eq:tp}
    \tp_n(\tbx) := \frac{1}{\tZ_n}\exp\left(-\frac{2n}{\sigma^2} \intene_n(\tbx)\right)
  \end{equation} 
  with respect to the Lebesgue measure $\dd\tbx$ on $\Mdn$.
\end{defi}

By definition, for all $n \geq 2$, the probability measure $\tP_n$ gives full weight to the subspace $\Mdn$.

\subsection{Two specific examples}\label{ss:examples} When $\intene_n$ is smooth enough to ensure the well-posedness of the system of stochastic differential equations
\begin{equation}\label{eq:sdeX}
  \dd X_i(t) = -n\nabla_{x_i} \intene_n\left(X_1(t), \ldots, X_n(t)\right)\dd t + \sigma \dd \beta_i(t),
\end{equation}
with $\beta_1, \ldots, \beta_n$ independent standard $\R^d$-valued Brownian motions, the Gibbs measure $\tP_n$ of Definition~\ref{defi:tP} is related to the long time behaviour of the diffusion process $(X_1(t), \ldots, X_n(t))_{t \geq 0}$. We first give two explicit examples of such processes, for which the energy functional satisfies Assumption~\eqref{ass:TI}.

\begin{ex}[\MV-model]
  Given a smooth, nonnegative and even \emph{interaction potential} $W : \R^d \to \R$, the energy functional
  \begin{equation}\label{eq:MVIntEne}
    \IntEne[\mu] = \frac{1}{2}\iint_{x,y \in \R^d} W(x-y)\dd\mu(x)\dd\mu(y)
  \end{equation}
  leads to the \emph{McKean-Vlasov particle system} without external potential
  \begin{equation*}
    \dd X_i(t) = -\frac{1}{n}\sum_{j=1}^n\nabla W\left(X_i(t)-X_j(t)\right)\dd t + \sigma \dd \beta_i(t).
  \end{equation*}
  This particle system arises, for instance, in the probabilistic approximation of the granular media equation~\cite{BenRoyTalVal98, BenRoyVal98, Mal03, CatGuiMal08}, for which the choice $W(x)=|x|^3$ is of particular physical interest~\cite{BenCagPul97, BenCagCarPul98}.
\end{ex}

\begin{ex}[\RB-model]
  In dimension $d=1$, given a $C^1$ and nonnegative function $B : [0,1] \to \R$ such that
  \begin{equation}\label{eq:RBti}
    B(0)=B(1)=0,
  \end{equation}
  the energy functional
  \begin{equation}\label{eq:RBIntEne}
    \IntEne[\mu] = \int_{x \in \R} B(F_\mu(x))\dd x,
  \end{equation}
  where $F_\mu$ denotes the cumulative distribution function of $\mu$, is associated with the system of \emph{rank-based interacting diffusions}
  \begin{equation}\label{eq:RBsde}
    \dd X_i(t) = \sum_{k=1}^n b_n(k)\ind{X_i(t) = X_{(k)}(t)}\dd t + \sigma \dd\beta_i(t),
  \end{equation}
  where for all $t \geq 0$, $X_{(1)}(t) \leq \cdots \leq X_{(n)}(t)$ denotes the order statistics of $X_1(t), \ldots, X_n(t)$, and
  \begin{equation*}
    b_n(k) := n\left(B\left(\frac{k}{n}\right)-B\left(\frac{k-1}{n}\right)\right).
  \end{equation*}
  This particle system serves as a model for large equity markets, and is also related to the probabilistic interpretation of nonlinear scalar conservation laws~\cite{Fer02, Rey17}. For the latter reason, we shall call $B$ a \emph{flux function}.
\end{ex}

\begin{rk}[Intersection between both classes of models]\label{rk:MVRB}
  Taking $d=1$ and $W(x)=|x|$ in the \MV-model yields the energy functional
  \begin{equation*}
    \IntEne[\mu] = \frac{1}{2}\iint_{x,y\in\R} |x-y|\dd\mu(x)\dd\mu(y) = \int_{x \in \R} F_\mu(x)(1-F_\mu(x))\dd x,
  \end{equation*}
  so that this model coincides with the \RB-model for $B(u)=u(1-u)$. 
\end{rk}

For both the \MV-model and the \RB-model, it is quickly observed that the centre of mass of the system
\begin{equation*}
  \Xi(t) := \frac{1}{n}\sum_{i=1}^n X_i(t)
\end{equation*}
is a Brownian motion in $\R^d$, which prevents $(X_1(t), \ldots, X_n(t))$ from converging to an equilibrium probability measure. Following the remark made in~\cite{Mal03} for the \MV-model and in~\cite{JouMal08, PalPit08} for the \RB-model, we define the diffusion process $(\tX_1(t), \ldots, \tX_n(t))_{t \geq 0}$ on the linear subspace $\Mdn$ by
\begin{equation}\label{eq:tX}
  \tX_i(t) := X_i(t) - \Xi(t),
\end{equation}
which describes the particle system \emph{seen from its centre of mass}. Under the assumptions of Lemma~\ref{lem:tP}, this process turns out to be reversible with respect to the Gibbs measure $\tP_n$.

\subsection{Free energy and large deviations} Under the assumptions of Lemma~\ref{lem:tP}, the central object of our study is the sequence of probability measures $\tPP_n$ defined by
\begin{equation}\label{eq:tPP}
  \tPP_n := \tP_n \circ \pi_n^{-1},
\end{equation}
which describe the distribution of the empirical measure of the particle system, seen from its centre of mass, at equilibrium. Notice that, for all $n \geq 2$, the restriction of $\pi_n$ to $\Mdn$ defines a continuous mapping from $\Mdn$ to either $\Ps(\R^d)$ or $\tPs_p(\R^d)$, for any $p \geq 1$; in particular, it is measurable for both the topology of weak convergence and the Wasserstein topology. As a consequence, for all $n \geq 2$, the probability measure $\tPP_n$ is well-defined on both the Borel $\sigma$-field of $\Ps(\R^d)$ and the Borel $\sigma$-field of $\tPs_p(\R^d)$. 

In order to study the large deviations of the sequence $\tPP_n$, we first introduce the following two functionals on $\Ps(\R^d)$.

\begin{defi}[Boltzmann's entropy]
  For all $\mu \in \Ps(\R^d)$, we let
  \begin{equation}\label{eq:defEnt}
    \Ent[\mu] := \int_{x \in \R^d} p(x) \log p(x) \dd x
  \end{equation}
  if $\mu \in \Ps_1(\R^d)$ and has a density $p(x)$ with respect to the Lebesgue measure on $\R^d$, and 
  \begin{equation*}
    \Ent[\mu] := +\infty
  \end{equation*}
  otherwise.
\end{defi}

\begin{rk}[On the moment condition]\label{rk:Ent}
  The requirement that $\mu \in \Ps_1(\R^d)$ ensures that the negative part of $p \log p$ is integrable~\cite[Remark~9.3.7, p.~212]{AmbGigSav08}, and therefore ensures that $\Ent[\mu]$ is well-defined as an element of $(-\infty,+\infty]$.
\end{rk}

\begin{defi}[Free energy]
  The \emph{free energy} $\FreEne$ associated with an energy functional $\IntEne : \Ps(\R^d) \to [0,+\infty]$ is defined by
  \begin{equation}\label{eq:FreEne}
    \FreEne[\mu] := \Ent[\mu] + \frac{2}{\sigma^2} \IntEne[\mu],
  \end{equation}
  for all $\mu \in \Ps(\R^d)$.
\end{defi}

\begin{rk}[Physical free energy]
  In statistical physics, $\sigma^2$ is usually assigned the value $2\kB T$, where $\kB$ is the Boltzmann constant and $T>0$ is the temperature, and Boltzmann's entropy is rather defined by $\Ent_{\mathrm{phys}} = -\kB \Ent$. Therefore, to be consistent with the classical definition of the free energy
  \begin{equation*}
    \FreEne_{\mathrm{phys}} = -T \Ent_{\mathrm{phys}} + \IntEne,
  \end{equation*}
  one should rather define the free energy to be worth $\frac{\sigma^2}{2} \Ent + \IntEne$. The difference with~\eqref{eq:FreEne} merely lies in the multiplicative constant, and we shall keep the latter definition as it alleviates some computations throughout the article.
\end{rk}

If the energy functional satisfies the assumptions of Lemma~\ref{lem:tP}, the free energy possesses the following properties.

\begin{lem}[Bounds on the free energy]\label{lem:C}
  Let $\IntEne : \Ps(\R^d) \to [0,+\infty]$ be an energy functional satisfying the assumptions of Lemma~\ref{lem:tP}. 
  \begin{enumerate}[label=(\roman*),ref=\roman*]
    \item\label{it:C:1} There exists $\mu \in \Ps(\R^d)$ such that $\FreEne[\mu] < +\infty$.
    \item\label{it:C:2} $\FreEne$ is bounded from below on $\Ps(\R^d)$.
  \end{enumerate}
\end{lem}

\begin{rk}\label{rk:Cfinite}
  It is easily checked that the uniform distribution on any compact set has a finite Boltzmann entropy, which by Assumption~\eqref{ass:F} yields the statement~\eqref{it:C:1} of Lemma~\ref{lem:C}.
\end{rk}

The statement~\eqref{it:C:2} of Lemma~\ref{lem:C} is proved in Subsection~\ref{ss:rate}.

Under the assumptions of Lemma~\ref{lem:C}, we may define
\begin{equation*}
  \FreEne_\star := \inf_{\mu \in \Ps(\R^d)} \FreEne[\mu] \in \R.
\end{equation*}
This quantity is sometimes referred to as \emph{Gibbs' free energy}~\cite{DawGar89}.

Before stating our first result, we introduce two further assumptions on the energy functional $\IntEne$:
\begin{enumerate}[label=(SH),ref=SH]
  \item\label{ass:SH} subhomogeneity: for all $\epsilon \in (0,1)$, for all $\bx \in \Rdn$, $(1-\epsilon)\intene_n(\bx) \geq \intene_n((1-\epsilon)\bx)$;
\end{enumerate}\begin{enumerate}[label=(CC),ref=CC]
  \item\label{ass:CC} chaos compatibility: for all $\mu \in \Ps(\R^d)$, if $(Y_n)_{n \geq 1}$ is a sequence of independent random variables with identical distribution $\mu$ on some probability space $(\Omega, \mathcal{A}, \Pr)$, then
  \begin{equation*}
    \lim_{n \to +\infty} \Exp\left[\intene_n(Y_1, \ldots, Y_n)\right] = \IntEne[\mu].
  \end{equation*}
\end{enumerate}

\begin{rk}[On Assumption~\eqref{ass:SH}]
  Unlike the remainder of the assumptions, Assumption~\eqref{ass:SH} is quite technical and is only employed once in the article, namely in the proof of the exponential estimates of Lemma~\ref{lem:expcontrol}. It may certainly be replaced by a variety of other similar assumptions, as long as they allow to obtain the same exponential estimates, but we believe that the present formulation achieves a reasonable balance between the generality of the models that it covers, and the relative simplicity of the computations that it requires to prove Lemma~\ref{lem:expcontrol}.
\end{rk}

\begin{rk}[On Assumptions~\eqref{ass:CC} and~\eqref{ass:LSC}]\label{rk:CCLSC}
  If $(Y_n)_{n \geq 1}$ is a sequence of independent random variables with identical distribution $\mu \in \Ps(\R^d)$, then by the Glivenko-Cantelli Lemma, the empirical measure $\pi_n(Y_1, \ldots, Y_n)$ converges to $\mu$ in $\Ps(\R^d)$, $\Pr$-almost surely. As a consequence, Assumption~\eqref{ass:LSC} and Fatou's Lemma yield
  \begin{equation*}
    \liminf_{n \to +\infty} \Exp\left[\intene_n(Y_1, \ldots, Y_n)\right] \geq \IntEne[\mu],
  \end{equation*}
  so that Assumption~\eqref{ass:CC} merely involves the limit superior of $\Exp[\intene_n(Y_1, \ldots, Y_n)]$.
\end{rk}

We are now ready to state the first main result of the article. We recall that, on a metric space, a \emph{good rate function} is a proper function with compact level sets, and refer to~\cite{DemZei10} for introductory material on large deviation principles.

\begin{theo}[LDP for $\tPP_n$ in Wasserstein spaces]\label{theo:Wass}
  Let $\IntEne : \Ps(\R^d) \to [0,+\infty]$ be an energy functional satisfying Assumptions~\eqref{ass:TI}, \eqref{ass:F}, \eqref{ass:LSC}, \eqref{ass:GC}, \eqref{ass:SH} and~\eqref{ass:CC}. If the index $\ell \geq 1$ given by Assumption~\eqref{ass:GC} is such that $\ell > 1$, then for all $p \in [1, \ell)$, the sequence $\tPP_n$ satisfies a large deviation principle on $\tPs_p(\R^d)$ with good rate function
  \begin{equation*}
    \tRate[\tmu] := \FreEne[\tmu] - \FreEne_\star.
  \end{equation*}
\end{theo}

Notice that the large deviation principle holds only in Wasserstein topologies with order strictly smaller than the index $\ell$ of Assumption~\eqref{ass:GC}, which for the \MV-model coincides with the order of polynomial growth of the interaction potential $W$. Furthermore, since the Wasserstein topology is stronger than the topology of weak convergence, the Contraction Principle~\cite[Theorem~4.2.1, p.~126]{DemZei10} implies that under the assumptions of Theorem~\ref{theo:Wass}, the large deviation principle for $\tPP_n$ also holds on the space $\Ps(\R^d)$, with good rate function $\Rate$ defined by
\begin{equation}\label{eq:Rate}
  \Rate[\mu] := \begin{cases}
    \tRate[\mu] & \text{if $\mu \in \tPs_1(\R^d)$,}\\
    +\infty & \text{otherwise.}
  \end{cases}
\end{equation}

As far as the role of the topology in the large deviation principle is concerned, a parallel can be drawn with Sanov's Theorem. Indeed, let $\QQ_n$ denote the law of the empirical measure of independent random variables in $\R^d$, with identical distribution $\nu$, where $\nu$ has a density proportional to $\exp(-2|x|^\ell/\sigma^2)$, with $\ell \geq 1$. The standard Sanov Theorem~\cite[Theorem~6.2.10, p.~263]{DemZei10} asserts that the sequence $\QQ_n$ satisfies a large deviation principle on $\Ps(\R^d)$, and it was proved by Wang, Wang and Wu~\cite{WanWanWu10} that, if $\ell > 1$, then the large deviation principle actually holds on $\Ps_p(\R^d)$, for $p \in [1,\ell)$ --- but \emph{not} for $p=\ell$.

Keeping the analogy between Sanov's Theorem and Theorem~\ref{theo:Wass} in mind, one may therefore wonder, if Assumption~\eqref{ass:GC} in the latter theorem is only satisfied with $\ell=1$, whether the large deviation principle continues to hold on $\Ps(\R^d)$, with the rate function $\Rate$ defined by~\eqref{eq:Rate}, for want of holding in a Wasserstein topology. We show that the answer is negative, by exhibiting an example for which the level sets of the function $\Rate$ fail to be compact on $\Ps(\R^d)$, which prevents the large deviation principle from holding. As should be clear from the example, this is related to the lack of continuity of the centering operator $\Tc$ on $\Ps(\R^d)$.

\begin{ex}[Counter-example to Theorem~\ref{theo:Wass} when $\ell=1$]
  We assume that $d=1$ and take the energy functional
  \begin{equation*}
    \IntEne[\mu] = \frac{1}{2}\iint_{x,y \in \R} |x-y|\dd\mu(x)\dd\mu(y)
  \end{equation*}
  of Remark~\ref{rk:MVRB}. It will be checked in Subsection~\ref{ss:RB} that this energy functional satisfies the assumptions of Theorem~\ref{theo:Wass}, except that Assumption~\eqref{ass:GC} is only satisfied with $\ell=1$. This in fact occurs for \emph{any} instance of the \RB-model, and not only for the case $B(u)=u(1-u)$ corresponding to the energy functional chosen here.
  
  Let $\varphi$ be the density of the standard Gaussian distribution on $\R$, and for all $\theta \in (0,1)$, let us define the density
  \begin{equation*}
    p_\theta(x) := (1-\theta)\varphi(x+1) + \theta\varphi(x-(1-\theta)/\theta).
  \end{equation*}
  For all $\theta \in (0,1)$, the probability measure $\mu_\theta$ with density $p_\theta$ is centered, and we have
  \begin{equation*}
    \Ent[\mu_\theta] \leq \Ent[\varphi] < +\infty
  \end{equation*}
  due to the convexity of $r \mapsto r \log r$, while
  \begin{equation*}
    \begin{aligned}
      \IntEne[\mu_\theta] &= \frac{1}{2}\iint_{x,y \in \R^d} |x-y|\dd\mu_\theta(x)\dd\mu_\theta(y)\\
      & \leq \int_{x \in \R} |x|\dd\mu_\theta(x)\\
      & \leq \int_{x \in \R} |x|\varphi(x)\dd x + 2(1-\theta) \leq \sqrt{\frac{2}{\pi}} + 2,
    \end{aligned}
  \end{equation*}
  where we have used the triangle inequality twice. As a consequence, the collection $\{\mu_\theta, \theta \in (0,1)\}$ is contained in a level set of the rate function $\Rate$ defined by~\eqref{eq:Rate}. But on the other hand, $\mu_\theta$ converges weakly, when $\theta$ vanishes, to the Gaussian distribution centered in $-1$, at which $\Rate$ takes the value $+\infty$. Therefore the level sets of $\Rate$ are not closed, whence not compact, in $\Ps(\R^d)$. 
\end{ex}

\subsection{Large deviations in the quotient space} Let us denote by $\bPs(\R^d)$ the orbit space of the group action
\begin{equation*}
  \left\{\begin{array}{ccc}
    \R^d \times \Ps(\R^d) & \to & \Ps(\R^d)\\
    (y,\mu) & \mapsto & \tau_y\mu
  \end{array}\right.
\end{equation*}
and define 
\begin{equation*}
  \rho : \left\{\begin{array}{ccc}
    \Ps(\R^d) & \to & \bPs(\R^d)\\
    \mu & \mapsto & \{\tau_y\mu: y \in \R\}
  \end{array}\right.
\end{equation*}
the associated \emph{orbit map}. The space $\bPs(\R^d)$ is endowed with the quotient topology, which is defined as the strongest topology making the map $\rho$ continuous. It is proved in Appendix~\ref{s:app} that this topology is metrisable. 

If a functional $\GenFun$ on $\Ps(\R^d)$ is translation invariant, then it is constant on orbits and we may define the functional $\bGenFun$ on $\bPs(\R^d)$ by
\begin{equation*}
  \bGenFun[\rho(\mu)] := \GenFun[\mu],
\end{equation*}
for any $\mu \in \Ps(\R^d)$. Under the assumptions of Lemma~\ref{lem:C}, the functionals $\IntEne$, $\Ent$ and $\FreEne$ are translation invariant, and it is immediate that
\begin{equation*}
  \inf_{\bmu \in \bPs(\R^d)} \bFreEne[\bmu] = \FreEne_\star.
\end{equation*}

For all $n \geq 2$, we define the probability measure
\begin{equation}\label{eq:bPP}
  \bPP_n := \tPP_n \circ \rho^{-1}
\end{equation}
on the Borel $\sigma$-field of $\bPs(\R^d)$. The next theorem is the second main result of this article.

\begin{theo}[LDP for $\bPP_n$ in the quotient space]\label{theo:quot}
  Let $\IntEne : \Ps(\R^d) \to [0,+\infty]$ be an energy functional satisfying Assumptions~\eqref{ass:TI}, \eqref{ass:F}, \eqref{ass:LSC}, \eqref{ass:GC}, \eqref{ass:SH} and~\eqref{ass:CC}. The sequence $\bPP_n$ satisfies a large deviation principle on $\bPs(\R^d)$ with good rate function
  \begin{equation*}
    \bRate[\bmu] := \bFreEne[\bmu] - \FreEne_\star.
  \end{equation*}
\end{theo}

Of course, in the case $\ell>1$, Theorem~\ref{theo:quot} can be obtained by contraction from Theorem~\ref{theo:Wass}, but we will not take advantage of this remark and we will rather prove both theorems simultaneously. 

\begin{rk}[Large deviations in $\bPs(\R^d)$]
  Let $\beta$ be a standard $\R^d$-valued Brownian motion, and consider the occupation measure
  \begin{equation*}
    L_t = \frac{1}{t} \int_{s=0}^t \delta_{\beta(s)}\dd s.
  \end{equation*}
  Because of the lack of ergodicity of the Brownian motion, the large deviations of $L_t$, when $t \to +\infty$, are not covered by the standard Donsker-Varadhan theory. Recently, Mukherjee and Varadhan~\cite{MukVar16} introduced a suitable compactification of the space $\bPs(\R^d)$, in which a large deviation principle can be stated for the orbit of $L_t$. This result also allows to get estimates on translation invariant functionals, such as probability measures on the space of sample paths with density proportional to
  \begin{equation*}
    \exp\left(t \iint_{x,y \in \R^d} W(x-y)\dd L_t(x)\dd L_t(y)\right)
  \end{equation*}
  with respect to the Wiener measure, for some interaction potential $W$.
  
  Although we rely on the same idea of working in the orbit space in order to compensate the lack of ergodicity of our original process, our topological construction is quite distinct. In particular, no compactification of the orbit space is required for Theorem~\ref{theo:quot} to hold.
\end{rk}

\subsection{Sketch of the proof of Theorems~\ref{theo:Wass} and~\ref{theo:quot}} Our two main theorems are proved simultaneously. In Section~\ref{s:conf}, we first state a large deviation principle for the law $\PP^\eta_n$ of the empirical measure of a system with energy functional $\IntEne$ and a confining functional $\PinEne^\eta$, the magnitude of which depends on a small parameter $\eta>0$. This result can be considered standard, and our proof closely follows the lines of~\cite[Theorem~1.5]{DupLasRam15}. For consistency when $\eta$ vanishes, we choose the external potential associated with $\PinEne^\eta$ to grow as $|x|^\ell$, where $\ell \geq 1$ is given by Assumption~\eqref{ass:GC}. As a result, the large deviation principle for $\PP^\eta_n$ holds on $\Ps(\R^d)$, and if $\ell>1$, on $\Ps_p(\R^d)$ for any $p \in [1,\ell)$. 

By contraction, we then obtain large deviation principles for the respective pushforward measures $\bPP^\eta_n$ and $\tPP^\eta_n$ of $\PP^\eta_n$ by $\rho$ and $\Tc$, respectively on $\bPs(\R^d)$, and if $\ell>1$, on $\tPs_p(\R^d)$ for any $p \in [1,\ell)$. The end of the proof, detailed in Section~\ref{s:pfmain}, then consists in checking that, when $\eta$ vanishes, $\bPP^\eta_n$ and $\tPP^\eta_n$ provide sufficiently good approximations of $\bPP_n$ and $\tPP_n$, at the level of large deviations. This part can be considered as the main original contribution of the article.

\subsection{Large deviations for the \texorpdfstring{\MV}{MV}-model and the \texorpdfstring{\RB}{RB}-model} We come back to the specific examples of the \MV-model and \RB-model introduced in Subsection~\ref{ss:examples}, and state large deviation principles for these models which come as corollaries of Theorems~\ref{theo:Wass} and~\ref{theo:quot}.

\subsubsection{\MV-model} Let $W : \R^d \to [0,+\infty)$ be an interaction potential which possesses the decomposition
\begin{equation}\label{eq:WsWf}
  W = W^\sharp + W^\flat,
\end{equation}
where the functions $W^\sharp : \R^d \to [0,+\infty)$ and $W^\flat : \R^d \to \R$ satisfy the following respective assumptions.
\begin{enumerate}[label=(\MV-$\sharp$), ref=\MV-$\sharp$]
  \item\label{ass:MVs} The function $W^\sharp$ is even, lower semicontinuous on $\R^d$, there exists $\ell \geq 1$ and $\kappa_\ell > 0$ such that, for all $x \in \R^d$, $W^\sharp(x) \geq 2\kappa_\ell|x|^\ell$, and for all $\epsilon \in (0,1)$, for all $x \in \R^d$, $(1-\epsilon)W^\sharp(x) \geq W^\sharp((1-\epsilon)x)$.
\end{enumerate}\begin{enumerate}[label=(\MV-$\flat$), ref=\MV-$\flat$]
  \item\label{ass:MVf} The function $W^\flat$ is even, continuous on $\R^d$ and, with $\ell \geq 1$ given by Assumption~\eqref{ass:MVs} on $W^\sharp$:
  \begin{itemize}
    \item if $\ell = 1$, then $W^\flat$ is bounded;
    \item if $\ell > 1$, then there exists $\ell' \in [0,\ell)$ such that $W^\flat(x)/(1+|x|^{\ell'})$ is bounded on $\R^d$.
  \end{itemize}
\end{enumerate}
Any polynomial function of $|x|$, with nonnegative but possibly fractional powers, degree larger or equal to $1$, and positive leading coefficient, satisfies this set of assumptions --- up to renormalisation of the constant term in order to ensure nonnegativity. This is in particular the case of the cubic potential $W(x)=|x|^3$ corresponding to the granular media equation~\cite{BenCagPul97, BenCagCarPul98}. However, singular potentials such as those involved in the particle approximation of the Keller-Segel equation~\cite{FouJou15,CatPed16}, or in the study of Coulomb gases~\cite{LebSer15}, do not satisfy our set of assumptions.

\begin{cor}[LDP for the \MV-model]\label{cor:MV}
  Let $W : \R^d \to [0,+\infty)$ be an interaction potential possessing the decomposition~\eqref{eq:WsWf}, with functions $W^\sharp$ and $W^\flat$ satisfying the respective Assumptions~\eqref{ass:MVs} and~\eqref{ass:MVf}. Let us define the energy functional $\IntEne$ by the identity~\eqref{eq:MVIntEne}. The sequence of associated probability measures $\tP_n$ is well-defined, and letting $\tPP_n$, $\bPP_n$ be defined by~\eqref{eq:tPP} and~\eqref{eq:bPP}, respectively, we have the following results.
  \begin{enumerate}[label=(\roman*),ref=\roman*]
    \item The sequence $\bPP_n$ satisfies a large deviation principle on $\bPs(\R^d)$ with good rate function $\bRate$ defined by Theorem~\ref{theo:quot}.
    \item If the index $\ell \geq 1$ of Assumptions~\eqref{ass:MVs} and~\eqref{ass:MVf} is such that $\ell>1$, then for all $p \in [1, \ell)$, the sequence $\tPP_n$ satisfies a large deviation principle on $\tPs_p(\R^d)$ with good rate function $\tRate$ defined by Theorem~\ref{theo:Wass}.
  \end{enumerate}
\end{cor}

The proof of Corollary~\ref{cor:MV} is presented in Subsection~\ref{ss:MV}. If $W^\flat \equiv 0$, then the energy functional $\IntEne$ actually satisfies the assumptions of Theorems~\ref{theo:Wass} and~\ref{theo:quot}, so that the result of Corollary~\ref{cor:MV} is straightforward. The case $W^\flat \not\equiv 0$ is treated as a perturbation of the previous case, thanks to the Laplace-Varadhan Lemma.

\subsubsection{\RB-model} Let $B : [0,1] \to [0,+\infty)$ be a $C^1$ flux function satisfying the condition~\eqref{eq:RBti}, which ensures that the energy functional $\IntEne$ defined by~\eqref{eq:RBIntEne} is not identically equal to $+\infty$. It is known~\cite{Rey17} that the condition 
\begin{equation}\label{eq:Oleinik}
  \forall u \in (0,1), \qquad B(u) > 0,
\end{equation}
which is called \emph{Oleinik's entropy condition} in the vocabulary of conservation laws, ensures the ergodicity of the centered particle system introduced in Subsection~\ref{ss:examples}. The combination of~\eqref{eq:RBti} and~\eqref{eq:Oleinik} implies that $B'(0) \geq 0 \geq B'(1)$, and the stronger condition
\begin{equation}\label{eq:Lax}
  B'(0) > 0 > B'(1),
\end{equation}
which is called \emph{Lax' entropy condition}, generally ensures better ergodic properties of both the particle system and its mean-field limit~\cite{JouMal08,JouRey13,JouRey15,Rey15}. Notice that if $B$ is assumed to be concave, then Oleinik's and Lax' conditions are equivalent, and hold as soon as $B$ is not identically zero.

We shall check in Subsection~\ref{ss:RB} that this set of conditions implies that the energy functional $\IntEne$ satisfies the assumptions of Theorem~\ref{theo:quot}, and in particular Assumption~\eqref{ass:GC} with $\ell=1$, which allows to define the sequence $\bPP_n$ associated with $\IntEne$ and leads to the following result.

\begin{cor}[LDP for the \RB-model]\label{cor:RB}
  Let $B : [0,1] \to [0,+\infty)$ be a $C^1$ flux function satisfying the conditions~\eqref{eq:RBti}, \eqref{eq:Oleinik} and~\eqref{eq:Lax}. Let $\IntEne$ be the energy functional associated with $B$ by~\eqref{eq:RBIntEne}. The sequence $\bPP_n$ associated with $\IntEne$ is well-defined, and it satisfies a large deviation principle on $\bPs(\R)$, with good rate function $\bRate$ given by Theorem~\ref{theo:quot}.
\end{cor}

In mathematical finance, systems of rank-based interacting diffusions are employed to model the evolution of the logarithmic capitalisations of stocks on an equity market~\cite{Fer02,BanFerKar05,JouRey15}. In Subsection~\ref{ss:SPT}, we present an application of Corollary~\ref{cor:RB} to the study of atypical capital distribution in this framework.

\section{Large deviations with a small external potential}\label{s:conf}

Throughout this section, $\IntEne : \Ps(\R^d) \to [0,+\infty]$ is an energy functional satisfying Assumptions~\eqref{ass:TI}, \eqref{ass:F}, \eqref{ass:LSC}, \eqref{ass:GC} and~\eqref{ass:CC}, and $\ell \geq 1$ is the index given by Assumption~\eqref{ass:GC}. We do not repeat these assumptions in the statements of our results.

We first introduce a few notations. For all $\eta > 0$, we define
\begin{equation}\label{eq:Veta}
  V^\eta(x) := \eta |x|^\ell,
\end{equation}
for all $x \in \R^d$, and let
\begin{equation}\label{eq:PinEneeta}
  \PinEne^\eta[\mu] := \int_{x \in \R^d} V^\eta(x)\dd\mu(x) \in [0,+\infty]
\end{equation}
for all $\mu \in \Ps(\R^d)$, as well as
\begin{equation*}
  V^\eta_n(\bx) := \PinEne^\eta[\pi_n(\bx)]
\end{equation*}
for all $\bx \in \Rdn$. Let 
\begin{equation}\label{eq:zeta}
  z^\eta := \int_{x \in \R^d} \exp\left(-\frac{2}{\sigma^2}V^\eta(x)\right) \dd x \in (0,+\infty),
\end{equation}
and let $\nu^\eta$ be the probability measure on $\R^d$ with density $(z^\eta)^{-1}\exp(-2V^\eta/\sigma^2)$ with respect to the Lebesgue measure on $\R^d$. 

\subsection{Relative entropy and Sanov's Theorem} We recall that the relative entropy of $\mu \in \Ps(\R^d)$ with respect to $\nu \in \Ps(\R^d)$ is defined by 
\begin{equation}\label{eq:RelEnt}
  \RelEnt[\mu|\nu] := \begin{cases}
    \displaystyle\int_{\R^d} \frac{\dd \mu}{\dd \nu} \log\left(\frac{\dd \mu}{\dd \nu}\right) \dd \nu & \text{if $\mu \ll \nu$,}\\
    +\infty & \text{otherwise}.
  \end{cases}
\end{equation}

The following lemma is straightforward.

\begin{lem}[From Boltzmann's entropy to relative entropy]\label{lem:entrelent}
  For all $\mu \in \Ps(\R^d)$,
  \begin{equation}\label{eq:entrelent}
    \RelEnt[\mu|\nu^\eta] = \Ent[\mu] + \frac{2}{\sigma^2}\PinEne^\eta[\mu] + \log z^\eta.
  \end{equation}
\end{lem}

The identity~\eqref{eq:entrelent} holds in $[0,+\infty]$, in the sense that if $\mu \not\in \Ps_\ell(\R^d)$, then $\RelEnt[\mu|\nu]=+\infty$; while if $\mu \in \Ps_\ell(\R^d)$, then $\Ent[\mu]$ and $\RelEnt[\mu|\nu]$ are simultaneously finite or equal to $+\infty$.

With the notations introduced above, let us define 
\begin{equation*}
  Q^\eta_n := (\nu^\eta)^{\otimes n}, \qquad \QQ^\eta_n := Q^\eta_n \circ \pi_n^{-1}.
\end{equation*}

\begin{prop}[Sanov's Theorem]\label{prop:Sanov}
  For all $\eta > 0$, the sequence $\QQ^\eta_n$ satisfies a large deviation principle on $\Ps(\R^d)$, with good rate function $\RelEnt[\cdot|\nu^\eta]$. If $\ell > 1$, the large deviation principle holds on $\Ps_p(\R^d)$ for all $p \in [1, \ell)$, with the same rate function.
\end{prop}
The statement of the large deviation principle on $\Ps(\R^d)$ is the usual formulation of Sanov's Theorem~\cite[Theorem~6.2.10, p.~263]{DemZei10}. Its extension to $\Ps_p(\R^d)$ is due to Wang, Wang and Wu~\cite{WanWanWu10}.

\subsection{Large deviations in the interacting case} Owing to Assumption~\eqref{ass:F}, we have
\begin{equation*}
  Z^\eta_n := \int_{\bx \in \Rdn} \exp\left(-\frac{2n}{\sigma^2}\left(V^\eta_n(\bx) + \intene_n(\bx)\right)\right) \dd\bx \in (0,+\infty),
\end{equation*}
and we denote by $P^\eta_n$ the probability measure on $\Rdn$ with density
\begin{equation}\label{eq:peta}
  p^\eta_n(\bx) := \frac{1}{Z^\eta_n}  \exp\left(-\frac{2n}{\sigma^2}\left(V^\eta_n(\bx) + \intene_n(\bx)\right)\right)
\end{equation}
with respect to the Lebesgue measure on $\Rdn$. We finally let 
\begin{equation}\label{eq:PPeta}
  \PP^\eta_n := P^\eta_n \circ \pi^{-1}_n,
\end{equation}
and define the free energy functional $\FreEne^\eta$ by
\begin{equation}\label{eq:FreEneeta}
  \FreEne^\eta[\mu] := \Ent[\mu] + \frac{2}{\sigma^2}\left(\PinEne^\eta[\mu] + \IntEne[\mu]\right).
\end{equation}

Large deviation principles for equilibrium mean-field systems with an external potential may be considered to be standard results in the literature~\cite{Leo87,DawGar89,ChaGozZit14,DupLasRam15}. We however give a complete proof of the next statement, which is adapted to our assumptions on the energy functional $\IntEne$, and follows closely the arguments of Dupuis, Laschos and Ramanan~\cite[Theorem~1.5]{DupLasRam15}.

\begin{prop}[LDP for the sequence $\PP^\eta_n$]\label{prop:Dupuis}
  For all $\eta > 0$, the sequence $\PP^\eta_n$ satisfies a large deviation principle on $\Ps(\R^d)$ with good rate function
  \begin{equation*}
    \Rate^\eta[\mu] := \FreEne^\eta[\mu] - \FreEne^\eta_\star,
  \end{equation*}
  where
  \begin{equation*}
    \FreEne^\eta_\star := \inf_{\mu \in \Ps(\R^d)} \FreEne^\eta[\mu] \in \R.
  \end{equation*}
  If $\ell > 1$, the large deviation principle holds on $\Ps_p(\R^d)$ for all $p \in [1, \ell)$, with the same rate function.
\end{prop}

Notice that the same arguments as in Remark~\ref{rk:Cfinite} show that $\FreEne^\eta_\star < +\infty$. On the other hand, combining Lemma~\ref{lem:entrelent} with~\eqref{eq:FreEneeta} yields
\begin{equation*}
  \FreEne^\eta[\mu] = \RelEnt[\mu|\nu^\eta] + \frac{2}{\sigma^2}\IntEne[\mu] - \log z^\eta,
\end{equation*}
so that the nonnegativity of both the relative entropy and the energy functional ensure that $\FreEne^\eta_\star > -\infty$.

We may now proceed to the proof of Proposition~\ref{prop:Dupuis}.

\begin{proof}
  The proof relies on the so-called \emph{weak convergence approach to large deviations} developed by Dupuis and Ellis~\cite{DupEll97}. Throughout the proof, we use the notation $\Pst(\R^d)$ to refer to either of the topological spaces $\Ps(\R^d)$ or $\Ps_p(\R^d)$, if $\ell > 1$ and $p \in [1,\ell)$. We recall that both spaces are Polish.
  
  As a first step, we invoke~\cite[Theorem~1.2.3, p.~7]{DupEll97} to reduce the proof of Proposition~\ref{prop:Dupuis} to the verification of the following two facts:
  \begin{enumerate}[label=(\roman*), ref=\roman*]
    \item the function $\Rate^\eta$ has compact level sets on $\Pst(\R^d)$;
    \item for any continuous and bounded functional $\GenFun: \Pst(\R^d) \to \R$, the \emph{Laplace principle}
    \begin{equation}\label{eq:Laplace}
      \lim_{n \to +\infty} -\frac{1}{n} \log \int_{\mu \in \Pst(\R^d)} \exp\left(-n \GenFun[\mu]\right) \dd\PP^\eta_n[\mu] = \inf_{\mu \in \Pst(\R^d)} \left\{\GenFun[\mu]+\Rate^\eta[\mu]\right\}
    \end{equation}
    holds.
  \end{enumerate}

  \emph{Proof of~(i).} Using Lemma~\ref{lem:entrelent}, we rewrite
  \begin{equation}\label{eq:refrate}
    \Rate^\eta[\mu] = \RelEnt[\mu|\nu^\eta] + \frac{2}{\sigma^2}\IntEne[\mu] - (\FreEne^\eta_\star+\log z^\eta),
  \end{equation}
  so that it suffices to show that $\RelEnt[\cdot|\nu^\eta] + \frac{2}{\sigma^2}\IntEne$ has compact level sets. As a consequence of Proposition~\ref{prop:Sanov}, $\RelEnt[\cdot|\nu^\eta]$ is a good rate function on $\Pst(\R^d)$ and therefore has compact level sets. Since the functional $\IntEne$ is nonnegative and satisfies Assumption~\eqref{ass:LSC}, then any level set of $\RelEnt[\cdot|\nu^\eta] + \frac{2}{\sigma^2}\IntEne$ is a closed subset of a level set of $\RelEnt[\cdot|\nu^\eta]$, and therefore is compact.
  
  \emph{Reformulation of~\eqref{eq:Laplace}.} Let us first remark that, on account of the definitions of $\PP^\eta_n$ and $\QQ^\eta_n$,
  \begin{equation}\label{eq:dPPdQQ}
    \frac{\dd \PP^\eta_n}{\dd \QQ^\eta_n}[\mu] \propto \exp\left(-\frac{2n}{\sigma^2}\IntEne[\mu]\right).
  \end{equation}
  As a consequence, the prelimit in~\eqref{eq:Laplace} rewrites
  \begin{equation}\label{eq:Laplace2}
    \begin{aligned}
      & -\frac{1}{n} \log \int_{\mu \in \Pst(\R^d)} \exp\left(-n \GenFun[\mu]\right) \dd\PP^\eta_n[\mu]\\
      & \qquad= -\frac{1}{n} \log \int_{\mu \in \Pst(\R^d)} \exp\left(-n \left(\GenFun[\mu]+\frac{2}{\sigma^2}\IntEne[\mu]\right)\right) \dd\QQ^\eta_n[\mu]\\
      & \qquad\quad + \frac{1}{n} \log \int_{\mu \in \Pst(\R^d)} \exp\left(-\frac{2n}{\sigma^2}\IntEne[\mu]\right) \dd\QQ^\eta_n[\mu],
    \end{aligned}
  \end{equation}
  so that it suffices to compute the limit of the first term in the right-hand side, and deduce the limit of the second by taking $\GenFun \equiv 0$. The computation of such quantities is typically the object of Varadhan's Lemma, which cannot be directly applied here since the functional $\IntEne$ is not assumed to be continuous and bounded.
  
  \emph{Lower bound in the Laplace principle.} Using the fact that $\IntEne$ is bounded from below and satisfies Assumption~\eqref{ass:LSC}, the combination of Proposition~\ref{prop:Sanov} with the variant of Varadhan's Lemma~\cite[Lemma~4.3.6, p.~138]{DemZei10} provides the lower bound
  \begin{equation}\label{eq:lowbLaplace}
    \begin{aligned}
      & \liminf_{n \to +\infty} -\frac{1}{n}\log \int_{\mu\in\Pst(\R^d)} \exp\left(-n \left(\GenFun[\mu]+\frac{2}{\sigma^2}\IntEne[\mu]\right)\right) \dd\QQ^\eta_n[\mu]\\
      & \qquad \geq \inf_{\mu \in \Pst(\R^d)} \left\{\GenFun[\mu] + \frac{2}{\sigma^2} \IntEne[\mu] + \RelEnt[\mu|\nu^\eta]\right\}.
    \end{aligned}
  \end{equation}
  
  \emph{Upper bound in the Laplace principle.} In order to obtain an upper bound of the same order as~\eqref{eq:lowbLaplace}, we first introduce a few notations. For all $\bx \in \Rdn$, we define
  \begin{equation*}
    \Psi_n(\bx) := -n\left(\GenFun[\pi_n(\bx)]+\frac{2}{\sigma^2}\IntEne[\pi_n(\bx)]\right),
  \end{equation*}
  and for all $M \geq 0$,
  \begin{equation*}
    \Psi^M_n(\bx) := \max(\Psi_n(\bx),-M).
  \end{equation*}
  The function $\Psi^M_n$ is measurable and bounded on $\Rdn$, so that the \emph{representation formula}~\cite[Proposition~1.4.2, p.~27]{DupEll97} --- or dually the Donsker-Varadhan variational characterisation of the relative entropy~\cite[Lemma~1.4.3, p.~29]{DupEll97} --- show that, for all probability measures $R_n$ on $\Rdn$,
  \begin{equation*}
    \RelEnt[R_n|Q^\eta_n] \geq \int_{\bx \in \Rdn} \Psi^M_n(\bx) \dd R_n(\bx) - \log \int_{\bx \in \Rdn} \exp\left(\Psi^M_n(\bx)\right) \dd Q^\eta_n(\bx),
  \end{equation*}
  where the definition of the relative entropy of probability measures on $\Rdn$ is the same as~\eqref{eq:RelEnt} for probability measures on $\R^d$. Using the trivial bound $\Psi^M_n(\bx) \geq \Psi_n(\bx)$ on the one hand, and the fact that since $\Psi_n$ is bounded from above on $\Rdn$, the Dominated Convergence Theorem yields
  \begin{equation*}
    \lim_{M \to +\infty} \int_{\bx \in \Rdn} \exp\left(\Psi^M_n(\bx)\right) \dd Q^\eta_n(\bx) = \int_{\bx \in \Rdn} \exp\left(\Psi_n(\bx)\right) \dd Q^\eta_n(\bx)
  \end{equation*}
  on the other hand, we deduce that
  \begin{equation*}
    \RelEnt[R_n|Q^\eta_n] \geq \int_{\bx \in \Rdn} \Psi_n(\bx) \dd R_n(\bx) - \log \int_{\bx \in \Rdn} \exp\left(\Psi_n(\bx)\right) \dd Q^\eta_n(\bx),
  \end{equation*}
  which rewrites
  \begin{equation}\label{eq:RnQn}
    \begin{aligned}
      & -\frac{1}{n}\log \int_{\mu \in \Pst(\R^d)} \exp\left(-n \left(\GenFun[\mu]+\frac{2}{\sigma^2}\IntEne[\mu]\right)\right) \dd\QQ^\eta_n[\mu]\\
      & \qquad \leq \int_{\bx\in\Rdn} \left(\GenFun[\pi_n(\bx)]+\frac{2}{\sigma^2}\IntEne[\pi_n(\bx)]\right)\dd R_n(\bx) + \frac{1}{n}\RelEnt[R_n|Q^\eta_n].
    \end{aligned}
  \end{equation}
  Let $\epsilon > 0$, and let $\mu_\epsilon \in \Pst(\R^d)$ be such that
  \begin{equation*}
    \GenFun[\mu_\epsilon] + \frac{2}{\sigma^2} \IntEne[\mu_\epsilon] + \RelEnt[\mu_\epsilon|\nu^\eta] \leq \inf_{\mu \in \Pst(\R^d)} \left\{\GenFun[\mu] + \frac{2}{\sigma^2} \IntEne[\mu] + \RelEnt[\mu|\nu^\eta]\right\} + \epsilon.
  \end{equation*}
  We evaluate the right-hand side of~\eqref{eq:RnQn} with $R_n = \mu_\epsilon^{\otimes n}$. On the one hand, it is easily seen that
  \begin{equation*}
    \frac{1}{n}\RelEnt[R_n|Q^\eta_n] = \RelEnt[\mu_\epsilon|\nu^\eta],
  \end{equation*}
  while on the other hand,
  \begin{equation*}
    \begin{aligned}
      & \int_{\bx \in \Rdn} \left(\GenFun[\pi_n(\bx)]+\frac{2}{\sigma^2}\IntEne[\pi_n(\bx)]\right)\dd R_n(\bx)\\
      & \qquad = \Exp\left[\GenFun[\pi_n(Y_1, \ldots, Y_n)] + \frac{2}{\sigma^2}\intene_n(Y_1, \ldots, Y_n)\right],
    \end{aligned}
  \end{equation*}
  where $Y_1, \ldots, Y_n$ are independent random variables in $\R^d$ with identical distribution $\mu_\epsilon$ on some probability space $(\Omega,\mathcal{A},\Pr)$. By Assumption~\eqref{ass:CC},
  \begin{equation*}
    \lim_{n \to +\infty} \Exp\left[\frac{2}{\sigma^2}\intene_n(Y_1, \ldots, Y_n)\right] = \IntEne[\mu_\epsilon],
  \end{equation*}
  whereas to justify the convergence of $\Exp[\GenFun[\pi_n(Y_1, \ldots, Y_n)]]$ to $\GenFun[\mu_\epsilon]$, we now show that
  \begin{equation}\label{eq:GlivenkoCantelli}
    \lim_{n \to +\infty} \pi_n(Y_1, \ldots, Y_n) = \mu_\epsilon, \qquad \text{$\Pr$-almost surely in $\Pst(\R^d)$,}
  \end{equation}
  and conclude by the Dominated Convergence Theorem using the fact that $\GenFun$ is continuous and bounded on $\Pst(\R^d)$. 
  \begin{itemize}
    \item If $\Pst(\R^d)$ refers to the topological space $\Ps(\R^d)$, then~\eqref{eq:GlivenkoCantelli} is the Glivenko-Cantelli Lemma. 
    \item If $\ell>1$ and $\Pst(\R^d)$ refers to the topological space $\Ps_p(\R^d)$, with $p \in [1,\ell)$, then by the strong Law of Large Numbers,
    \begin{equation*}
      \lim_{n \to +\infty} \frac{1}{n}\sum_{i=1}^n |Y_i|^p = \int_{x \in \R^d} |x|^p\mu_\epsilon(\dd x), \qquad \text{$\Pr$-almost surely,}
    \end{equation*}
    which, combined with the Glivenko-Cantelli Lemma, implies the $\Pr$-almost sure convergence in $\Ps_p(\R^d)$ of $\pi_n(Y_1, \ldots, Y_n)$ to $\mu_\epsilon$~\cite[Definition~6.8, p.~96]{Vil09}.
  \end{itemize}
  As a consequence, we finally get
  \begin{equation}\label{eq:upbLaplace}
    \begin{aligned}
      & \limsup_{n \to +\infty} -\frac{1}{n}\log \int_{\mu \in \Pst(\R^d)} \exp\left(-n \left(\GenFun[\mu]+\frac{2}{\sigma^2}\IntEne[\mu]\right)\right) \dd\QQ^\eta_n[\mu]\\
      & \qquad \leq  \left\{\GenFun[\mu_\epsilon] + \frac{2}{\sigma^2} \IntEne[\mu_\epsilon] + \RelEnt[\mu_\epsilon|\nu^\eta]\right\}\\
      & \qquad \leq  \inf_{\mu \in \Pst(\R^d)} \left\{\GenFun[\mu] + \frac{2}{\sigma^2} \IntEne[\mu] + \RelEnt[\mu|\nu^\eta]\right\} + \epsilon.
    \end{aligned}
  \end{equation}
  
  \emph{Conclusion of the proof.} Letting $\epsilon \dto 0$ in~\eqref{eq:upbLaplace} and combining the latter inequality with~\eqref{eq:lowbLaplace}, we conclude that
  \begin{equation*}
    \begin{aligned}
      & \lim_{n \to +\infty} -\frac{1}{n}\log \int_{\mu \in \Pst(\R^d)} \exp\left(-n \left(\GenFun[\mu]+\frac{2}{\sigma^2}\IntEne[\mu]\right)\right) \dd\QQ^\eta_n[\mu]\\
      & \qquad = \inf_{\mu \in \Pst(\R^d)} \left\{\GenFun[\mu] + \frac{2}{\sigma^2} \IntEne[\mu] + \RelEnt[\mu|\nu^\eta]\right\}, 
    \end{aligned}
  \end{equation*}
  so that, taking~\eqref{eq:Laplace2} into account,
  \begin{equation*}
    \begin{aligned}
      & \lim_{n \to +\infty} -\frac{1}{n} \log \int_{\mu\in\Pst(\R^d)} \exp\left(-n \GenFun[\mu]\right) \dd\PP^\eta_n[\mu]\\
      & = \inf_{\mu \in \Pst(\R^d)} \left\{\GenFun[\mu] + \frac{2}{\sigma^2} \IntEne[\mu] + \RelEnt[\mu|\nu^\eta]\right\} - \inf_{\mu \in \Pst(\R^d)} \left\{\frac{2}{\sigma^2} \IntEne[\mu] + \RelEnt[\mu|\nu^\eta]\right\}.
    \end{aligned}
  \end{equation*}
  By~\eqref{eq:refrate}, the right-hand side above rewrites $\inf_{\mu \in \Pst(\R^d)} \{\GenFun[\mu] + \Rate^\eta[\mu]\}$, which yields~\eqref{eq:Laplace} and completes the proof.
\end{proof}

\subsection{The measures \texorpdfstring{$\bPP^\eta_n$}{.} and \texorpdfstring{$\tPP^\eta_n$}{.}} Let us define the functional $\vartheta : \Ps(\R^d) \to [0,+\infty]$ by
\begin{equation*}
  \vartheta[\mu] := \inf_{y \in \R^d} \int_{x \in \R^d} |x+y|^\ell \dd\mu(x).
\end{equation*}
Notice that $\vartheta[\mu] < +\infty$ if and only if $\mu \in \Ps_\ell(\R^d)$, and that $\vartheta$ is translation invariant. 

For all $\eta > 0$, we define the probability measures $\bPP^\eta_n$ and $\tPP^\eta_n$, respectively on the Borel $\sigma$-fields of the topological spaces $\bPs(\R^d)$ and $\tPs_p(\R^d)$, for any $p \geq 1$, by the identities
\begin{equation}\label{eq:bPPetatPPeta}
  \bPP^\eta_n := \PP^\eta_n \circ \rho^{-1}, \qquad \tPP^\eta_n := \PP^\eta_n \circ \Tc^{-1}.
\end{equation}

Since the operators $\rho : \Ps(\R^d) \to \bPs(\R^d)$ and $\Tc : \Ps_p(\R^d) \to \tPs_p(\R^d)$ are continuous, the following result is obtained from Proposition~\ref{prop:Dupuis} by means of the Contraction Principle~\cite[Theorem~4.2.1, p.~126]{DemZei10}.

\begin{cor}[LDP for $\bPP^\eta_n$ and $\tPP^\eta_n$]\label{cor:Contraction}
  For all $\eta > 0$, the sequence $\bPP^\eta_n$ satisfies a large deviation principle on $\bPs(\R^d)$ with good rate function
  \begin{equation*}
    \bRate^\eta[\bmu] := \bEnt[\bmu] + \frac{2}{\sigma^2}\left(\eta \bar{\vartheta}[\bmu] + \bIntEne[\bmu]\right) - \FreEne^\eta_\star.
  \end{equation*}
  In addition, if $\ell > 1$, then for all $p \in [1,\ell)$, for all $\eta > 0$, the sequence $\tPP^\eta_n$ satisfies a large deviation principle on $\tPs_p(\R^d)$ with good rate function
  \begin{equation*}
    \tRate^\eta[\tmu] := \Ent[\tmu] + \frac{2}{\sigma^2}\left(\eta \vartheta[\tmu] + \IntEne[\tmu]\right) - \FreEne^\eta_\star.
  \end{equation*}
\end{cor}

\subsection{Alternative expression for \texorpdfstring{$\tPP^\eta_n$}{.}} We denote by $\tc_n$ the orthogonal projection of $\Rdn$ onto the subspace $\Mdn$, and for all $\eta > 0$, we define the probability measure $\tP^\eta_n$ on $\Rdn$ by
\begin{equation}\label{eq:tPeta}
  \tP^\eta_n := P^\eta_n \circ \tc_n^{-1}.
\end{equation}
Notice that $\tP^\eta_n(\Mdn)=1$. We also define the function $\hat{V}_n^\eta : \Mdn \to \R$ by the identity
\begin{equation}\label{eq:hatVeta}
  \exp\left(-\frac{2n}{\sigma^2}\hat{V}_n^\eta(\tbx)\right) = \int_{\zeta \in \R^d} \exp\left(-\frac{2n}{\sigma^2}V^\eta_n(\tbx+\vec{\zeta})\right)\dd\zeta,
\end{equation}
where for all $\zeta \in \R^d$, we denote by $\vec{\zeta} = (\zeta, \ldots, \zeta)$ the corresponding element of $\Rdn$. 

\begin{lem}[Relation between $\tPP^\eta_n$ and $\tP^\eta_n$]\label{lem:tPPeta-alt}
  For all $\eta > 0$,
  \begin{equation}\label{eq:tZeta}
    \tZ^\eta_n := \int_{\tbx\in\Mdn} \exp\left(-\frac{2n}{\sigma^2}\left(\hat{V}_n^\eta(\tbx) + W_n(\tbx)\right)\right)\dd\tbx \in (0,+\infty),
  \end{equation}
  and the probability measure $\tP^\eta_n$ defined by~\eqref{eq:tPeta} possesses the density
  \begin{equation*}
    \tp^\eta_n(\tbx) := \frac{1}{\tZ^\eta_n} \exp\left(-\frac{2n}{\sigma^2}\left(\hat{V}_n^\eta(\tbx) + W_n(\tbx)\right)\right)
  \end{equation*}
  with respect to the Lebesgue measure $\dd\tbx$ on $\Mdn$. Besides, the probability measure $\tPP^\eta_n$ defined by~\eqref{eq:bPPetatPPeta} satisfies
  \begin{equation}\label{eq:tPPetatPeta}
    \tPP^\eta_n = \tP^\eta_n \circ \pi_n^{-1}.
  \end{equation}
\end{lem}
\begin{proof}
  Let $B$ be a Borel subset of $\Rdn$. By~\eqref{eq:tPeta} and~\eqref{eq:peta},
  \begin{equation*}
    \tP^\eta_n(B) = \frac{1}{Z^\eta_n} \int_{\bx\in\Rdn} \ind{\tc_n(\bx)\in B} \exp\left(-\frac{2n}{\sigma^2}\left(V^\eta_n(\bx) + \intene_n(\bx)\right)\right)\dd\bx.
  \end{equation*}
  Any $\bx \in \Rdn$ admits the orthogonal decomposition $\bx = \tbx + \vec{\zeta}$, with $\tbx = \tc_n(\bx) \in \Mdn$ and $\vec{\zeta} = (\zeta, \ldots, \zeta)$ for some $\zeta \in \R^d$. As a consequence, $\tP^\eta_n(B)$ rewrites
  \begin{equation*}
    \begin{aligned}
      & \frac{\sqrt{n}^d}{Z^\eta_n} \iint_{\tbx\in\Mdn,\zeta\in\R^d} \ind{\tbx\in B} \exp\left(-\frac{2n}{\sigma^2}\left(V^\eta_n(\tx_i+\vec{\zeta}) + \intene_n(\tbx+\vec{\zeta})\right)\right)\dd\tbx\dd\zeta\\
      & \qquad = \frac{\sqrt{n}^d}{Z^\eta_n} \int_{\tbx\in\Mdn} \ind{\tbx\in B} \exp\left(-\frac{2n}{\sigma^2}\left(\hat{V}^\eta(\tbx) + \intene_n(\tbx)\right)\right)\dd\tbx,
    \end{aligned}
  \end{equation*}
  where we have used the fact that $\intene_n(\tbx+\vec{\zeta})=\intene_n(\tbx)$, thanks to Assumption~\eqref{ass:TI}, and the definition~\eqref{eq:hatVeta} of $\hat{V}_n^\eta$. This shows~\eqref{eq:tZeta} and the fact that $\tP^\eta_n$ possesses the density $\tp^\eta_n(\tbx)$. Last,~\eqref{eq:tPPetatPeta} follows from the elementary relation $\Tc \circ \pi_n = \pi_n \circ \tc_n$ on $\Rdn$.
\end{proof}

In Section~\ref{s:pfmain}, we shall rely on the following bounds on the function $\hat{V}_n^\eta$.

\begin{lem}[Bounds on $\hat{V}_n^\eta$]\label{lem:hatVeta-bounds}
  Let $n \geq 2$ and $\eta > 0$. For all $\tbx \in \Mdn$,
  \begin{equation*}
    \frac{\sigma^2}{2n}\log\frac{n^{d/\ell}}{z^\eta} \leq \hat{V}_n^\eta(\tbx) \leq \frac{2^{\ell-1}\eta}{n}\sum_{i=1}^n |\tx_i|^\ell + \frac{\sigma^2}{2n}\log\frac{(2^{\ell-1}n)^{d/\ell}}{z^\eta},
  \end{equation*}
  where we recall the definition~\eqref{eq:zeta} of $z^\eta$.
\end{lem}
\begin{proof}
  The upper bound follows from the convexity inequality
  \begin{equation*}
    |\tx_i + \zeta|^\ell \leq 2^{\ell-1}|\tx_i|^\ell + 2^{\ell-1}|\zeta|^\ell, \qquad 1 \leq i \leq n,
  \end{equation*}
  while the lower bound follows from Jensen's inequality
  \begin{equation*}
    |\zeta|^\ell = \left|\frac{1}{n}\sum_{i=1}^n \tx_i+\zeta\right|^\ell \leq \frac{1}{n}\sum_{i=1}^n |\tx_i+\zeta|^\ell,
  \end{equation*}
  since $\tbx \in \Mdn$.
\end{proof}

\section{Proof of Theorems~\ref{theo:Wass} and~\ref{theo:quot}}\label{s:pfmain}

This section is dedicated to the proof of the large deviation principles contained in Theorems~\ref{theo:Wass} and~\ref{theo:quot}. We first check in Subsection~\ref{ss:rate} that, under the respective assumptions of these theorems, the functionals $\tRate$ and $\bRate$ are good rate functions. In Subsection~\ref{ss:expcomp}, we obtain auxiliary results on the respective approximation of $\tPP_n$ and $\bPP_n$ by the measures $\tPP^\eta_n$ and $\bPP^\eta_n$ introduced in Section~\ref{s:conf}. These results allow us to prove large deviation upper and lower bounds in Subsection~\ref{ss:bounds}, thereby completing the proof of Theorems~\ref{theo:Wass} and~\ref{theo:quot}.

\subsection{Rate functions}\label{ss:rate} The purpose of this subsection is to prove the following result.

\begin{lem}[Goodness of rate functions]\label{lem:goodrate}
  Under the assumptions of Lemma~\ref{lem:C}, the functional $\bFreEne$ has compact level sets on $\bPs(\R^d)$, and if the index $\ell \geq 1$ given by Assumption~\eqref{ass:GC} is such that $\ell > 1$, then for all $p \in [1, \ell)$, the functional $\FreEne$ has compact level sets on $\tPs_p(\R^d)$.
\end{lem}

Combining the results of Lemmas~\ref{lem:C} and~\ref{lem:goodrate}, we conclude that, under the respective assumptions of Theorems~\ref{theo:Wass} and~\ref{theo:quot}, the functionals $\tRate$ and $\bRate$ are good rate functions, respectively on $\tPs_p(\R^d)$ and $\bPs(\R^d)$. We first state an auxiliary result.

\begin{lem}[Level sets on $\Ps(\R^d)$]\label{lem:lsPs}
  Under the assumptions of Lemma~\ref{lem:C}, for all $a \in \R$, the set
  \begin{equation*}
    A := \left\{\mu \in \Ps(\R^d) : \FreEne[\mu] \leq a\right\}
  \end{equation*}
  is closed in $\Ps(\R^d)$. Besides, letting $\ell \geq 1$ be given by Assumption~\eqref{ass:GC}, we have $A \subset \Ps_\ell(\R^d)$ and there exists $a' \in \R$ such that
  \begin{equation}\label{eq:momentcontrol}
    \forall \mu \in A, \qquad \int_{x \in \R^d} |x|^\ell \dd\Tc\mu(x) \leq a'.
  \end{equation}
\end{lem}
\begin{proof}
  Since, by Remark~\ref{rk:Ent}, neither $\Ent[\mu]$ nor $\IntEne[\mu]$ can take the value $-\infty$, any $\mu \in A$ satisfies $\Ent[\mu] < +\infty$ and $\IntEne[\mu] < +\infty$, which by Assumption~\eqref{ass:GC} ensures that $A \subset \Ps_\ell(\R^d)$.
  
  Let us now fix $\mu \in A$ and define $\tmu = \Tc\mu \in \tPs_\ell(\R^d)$. For all $\eta > 0$, we recall the definitions of $z^\eta$ and $\nu^\eta$ from Section~\ref{s:conf}. By the translation invariance of $\FreEne$, Lemma~\ref{lem:entrelent} and the definition~\eqref{eq:PinEneeta} of $\PinEne^\eta$,
  \begin{equation*}
    \FreEne[\mu] = \FreEne[\tmu] = \RelEnt[\tmu|\nu^\eta] + \frac{2}{\sigma^2}\left(\IntEne[\tmu]-\eta\int_{x \in \R^d}|x|^\ell\dd\tmu(x)\right) - \log z^\eta.
  \end{equation*}
  Using the fact that the relative entropy is nonnegative and then Assumption~\eqref{ass:GC}, we deduce that
  \begin{equation}\label{eq:borneinfF}
    \FreEne[\mu] \geq \frac{2}{\sigma^2}(\kappa_\ell-\eta)\int_{x \in \R^d}|x|^\ell\dd\tmu(x) - \log z^\eta,
  \end{equation}
  so that taking $\eta = \kappa_\ell/2$ and recalling that $\mu \in A$ yields
  \begin{equation*}
    \int_{x \in \R^d}|x|^\ell\dd\tmu(x) \leq \frac{\sigma^2}{\kappa_\ell}\left(a + \log z^\eta\right) =: a',
  \end{equation*}
  which provides~\eqref{eq:momentcontrol}.
  
  In order to show that $A$ is closed in $\Ps(\R^d)$, let us take a sequence $(\mu_n)_{n \geq 1}$ in $A$, which converges to some $\mu$ in $\Ps(\R^d)$, and prove that
  \begin{equation*}
    \liminf_{n \to +\infty} \FreEne[\mu_n] \geq \FreEne[\mu]
  \end{equation*}
  which implies $\mu \in A$. As a first step, we note that, according to the first part of the proof, $\mu_n \in \Ps_\ell(\R^d)$ for all $n \geq 1$, which allows us to define $\tmu_n = \Tc\mu_n$ and notice that $\FreEne[\tmu_n] = \FreEne[\mu_n]$; besides, by~\eqref{eq:momentcontrol}, the sequence of $\ell$-th order moments of $\tmu_n$ is bounded. Since the functional $\IntEne$ is nonnegative, the sequence $\FreEne[\tmu_n]$ is also bounded. Denoting by $\tp_n$ the density of $\tmu_n$, we then obtain from standard arguments~\cite[pp.~7-8]{JorKinOtt98} the existence of a probability density $\tilde{q}$ toward which $\tp_n$ converges weakly in $L^1(\R^d)$, at least along a subsequence, and such that
  \begin{equation*}
    \liminf_{n \to +\infty} \FreEne[\tmu_n] \geq \FreEne[\tilde{\nu}],
  \end{equation*}
  where we denote by $\tilde{\nu}$ the probability measure with density $\tilde{q}$. Finally, since the orbit map $\rho : \Ps(\R^d) \to \bPs(\R^d)$ is continuous, the series of identities
  \begin{equation*}
    \rho(\mu) = \lim_{n \to +\infty} \rho(\mu_n) = \lim_{n \to +\infty} \rho(\tmu_n) = \rho(\tilde{\nu})
  \end{equation*}
  in $\bPs(\R^d)$ implies that $\FreEne[\mu] = \FreEne[\tilde{\nu}]$, whence the conclusion.
\end{proof}

The inequality~\eqref{eq:borneinfF} shows that $\FreEne$ is bounded from below on $\Ps(\R^d)$, which proves the statement~\eqref{it:C:2} of Lemma~\ref{lem:C}. We may now complete the proof of Lemma~\ref{lem:goodrate}.

\begin{proof}[Proof of Lemma~\ref{lem:goodrate}]
  We fix $a \in \R$ and first prove that the set
  \begin{equation*}
    \bar{A} := \left\{\bmu \in \bPs(\R^d) : \bFreEne[\bmu] \leq a\right\}
  \end{equation*}
  is compact in $\bPs(\R^d)$. By Lemma~\ref{lem:openclosedsets}, this set is closed if and only if $\rho^{-1}(\bar{A})$ is closed in $\Ps(\R^d)$, which is the case since $\rho^{-1}(\bar{A})$ is easily seen to coincide with the set $A$ of Lemma~\ref{lem:lsPs}. We now proceed to show that this set is sequentially compact. Let $(\bmu_n)_{n \geq 1}$ be a sequence of elements of $\bar{A}$. By Lemma~\ref{lem:lsPs}, for all $n \geq 1$ there exists $\tmu_n \in A \cap \tPs_\ell(\R^d)$ such that $\rho(\tmu_n) = \bmu_n$, and we have the moment control
  \begin{equation}\label{eq:momentcontrol2}
    \forall n \geq 1, \qquad \int_{x \in \R^d} |x|^\ell\dd\tmu_n(x) \leq a'
  \end{equation}
  given by~\eqref{eq:momentcontrol}. Markov's inequality implies that the sequence $(\tmu_n)_{n \geq 1}$ is tight, so that by Prohorov's Theorem~\cite[Theorem~5.1, p.~59]{Bil99}, it possesses a converging subsequence. The continuity of the map $\rho$ then ensures that the sequence $(\bmu_n)_{n \geq 1}$ possesses a converging subsequence as well, which shows the sequential compactness of $\bar{A}$. Since we prove in Lemma~\ref{lem:bdP} that the quotient topology on $\bPs(\R^d)$ is metrisable,~\cite[Theorem~B.2, p.~345]{DemZei10} allows us to conclude that $\bar{A}$ is compact and obtain the first part of Lemma~\ref{lem:goodrate}.
  
  We now assume that $\ell > 1$, fix $p \in [1,\ell)$, and prove that the set
  \begin{equation*}
    \tilde{A} := \left\{\tmu \in \tPs_p(\R^d) : \FreEne[\tmu] \leq a\right\} = A \cap \tPs_p(\R^d)
  \end{equation*}
  is compact in $\tPs_p(\R^d)$. Since the Wasserstein topology is stronger than the topology of weak convergence, Lemma~\ref{lem:lsPs} implies that $A$ is closed in $\Ps_p(\R^d)$, and therefore $\tilde{A}$ is closed in $\tPs_p(\R^d)$. Now for all sequences $(\tmu_n)_{n \geq 1}$ of elements of $\tilde{A}$, the moment control~\eqref{eq:momentcontrol2} ensures that $(\tmu_n)_{n \geq 1}$ possesses a subsequence, that we still index by $n$ for convenience, which converges to some $\mu$ in $\Ps(\R^d)$. To prove that the convergence actually holds in $\tPs_p(\R^d)$, we remark that since $p < \ell$, the moment control~\eqref{eq:momentcontrol2} also ensures the uniform integrability of the $p$-th order moment of $\tmu_n$, so that by~\cite[Definition~6.8, p.~96]{Vil09}, $\tmu_n$ converges to $\mu$ in $\tPs_p(\R^d)$, therefore $\tilde{A}$ is sequentially compact in $\tPs_p(\R^d)$. By~\cite[Theorem~B.2, p.~345]{DemZei10} again, we conclude that $\tilde{A}$ is compact in $\tPs_p(\R^d)$, whence the second part of Lemma~\ref{lem:goodrate}.
\end{proof}

\subsection{Exponential comparisons}\label{ss:expcomp} This subsection contains two auxiliary results which will be used in the proof of the large deviation upper and lower bounds.

\begin{lem}[Exponential tilting of $\tP^\eta_n$]\label{lem:PPetaPP}
  Let $\IntEne : \Ps(\R^d) \to [0,+\infty]$ be an energy functional satisfying Assumptions~\eqref{ass:TI}, \eqref{ass:F}, \eqref{ass:LSC} and \eqref{ass:GC}, and let $\eta > 0$.
  \begin{enumerate}[label=(\roman*),ref=\roman*]
    \item For all $p \geq 1$, for all Borel sets $\tilde{B}$ of $\tPs_p(\R^d)$,
    \begin{equation}\label{eq:tPPetatPP:0}
      \tPP^\eta_n(\tilde{B}) = \int_{\tbx \in \Mdn} \ind{\pi_n(\tbx) \in \tilde{B}}\dd\tP^\eta_n(\tbx),
    \end{equation}
    and
    \begin{equation}\label{eq:tPPetatPP}
      \tPP_n(\tilde{B}) = \frac{\tZ^\eta_n}{\tZ_n}\int_{\tbx \in \Mdn} \ind{\pi_n(\tbx) \in \tilde{B}} \exp\left(\frac{2n}{\sigma^2}\hat{V}^\eta_n(\tbx)\right)\dd\tP^\eta_n(\tbx).
    \end{equation}
    \item For all Borel sets $\bar{B}$ of $\bPs(\R^d)$,
    \begin{equation}\label{eq:bPPetabPP:0}
      \bPP^\eta_n(\bar{B}) = \int_{\tbx \in \Mdn} \ind{\rho(\pi_n(\tbx)) \in \bar{B}} \dd\tP^\eta_n(\tbx),
    \end{equation}
    and
    \begin{equation}\label{eq:bPPetabPP}
      \bPP_n(\bar{B}) = \frac{\tZ^\eta_n}{\tZ_n}\int_{\tbx \in \Mdn} \ind{\rho(\pi_n(\tbx)) \in \bar{B}} \exp\left(\frac{2n}{\sigma^2}\hat{V}^\eta_n(\tbx)\right)\dd\tP^\eta_n(\tbx).
    \end{equation}
  \end{enumerate}
\end{lem}
\begin{proof}
  We first address the proof of the identities~\eqref{eq:tPPetatPP:0} and~\eqref{eq:bPPetabPP:0}. The equality~\eqref{eq:tPPetatPP:0} is a straightforward consequence of the definition~\eqref{eq:bPPetatPPeta} of $\tPP^\eta_n$. To check the validity of~\eqref{eq:bPPetabPP:0}, we recall that the respective definitions~\eqref{eq:bPPetatPPeta} and~\eqref{eq:PPeta} of $\bPP^\eta_n$ and $\PP^\eta_n$ yield
  \begin{equation}\label{eq:PPetaPP:0}
    \bPP^\eta_n = \PP^\eta_n \circ \rho^{-1} = P^\eta_n \circ \pi_n^{-1} \circ \rho^{-1}.
  \end{equation}
  Besides, since for all $\bx \in \Rdn$,
  \begin{equation*}
    \pi_n(\tc_n(\bx)) = \tau_\xi\pi_n(\bx), \qquad \xi := \frac{1}{n}\sum_{i=1}^n x_i, 
  \end{equation*}
  we have $\rho \circ \pi_n = \rho \circ \pi_n \circ \tc_n$ on $\Rdn$. Hence we may substitute $\pi_n^{-1} \circ \rho^{-1}$ with $\tc_n^{-1} \circ \pi_n^{-1} \circ \rho^{-1}$ in~\eqref{eq:PPetaPP:0} to obtain
  \begin{equation*}
    \bPP^\eta_n = P^\eta_n \circ \tc_n^{-1} \circ \pi_n^{-1} \circ \rho^{-1} = \tP^\eta_n \circ \pi_n^{-1} \circ \rho^{-1},
  \end{equation*}
  thanks to~\eqref{eq:tPeta}. This equality immediately leads to~\eqref{eq:bPPetabPP:0}.
  
  We now address the proof of~\eqref{eq:tPPetatPP} and~\eqref{eq:bPPetabPP}. For all $p \geq 1$, for all Borel sets $\tilde{B}$ of $\tPs_p(\R^d)$, \eqref{eq:tPP} yields
  \begin{equation*}
    \tPP_n(\tilde{B}) = \frac{1}{\tZ_n}\int_{\tbx \in \Mdn} \ind{\pi_n(\tbx) \in \tilde{B}} \exp\left(-\frac{2n}{\sigma^2}W_n(\tbx)\right)\dd\tbx,
  \end{equation*}
  so that~\eqref{eq:tPPetatPP} follows from Lemma~\ref{lem:tPPeta-alt}. Likewise, for all Borel sets $\bar{B}$ of $\bPs(\R^d)$,~\eqref{eq:bPPetabPP} is obtained by the same chain of arguments, but starting with~\eqref{eq:bPP} in place of~\eqref{eq:tPP}.
\end{proof}

\begin{lem}[Exponential moment control]\label{lem:expcontrol}
  Let $\IntEne : \Ps(\R^d) \to [0,+\infty]$ be an energy functional satisfying Assumptions~\eqref{ass:TI}, \eqref{ass:F}, \eqref{ass:LSC}, \eqref{ass:GC} and~\eqref{ass:SH}. For all $q \in [1,+\infty)$,
  \begin{equation}\label{eq:expcontrol}
    \limsup_{\eta \dto 0} \limsup_{n \to +\infty} \frac{1}{n} \log \int_{\tbx \in \Mdn} \exp\left(\frac{2nq}{\sigma^2}\hat{V}^\eta_n(\tbx)\right)\dd\tP^\eta_n(\tbx) \leq 0.
  \end{equation}
\end{lem}
\begin{proof}
  Let us fix $q \in [1,+\infty)$ and $\epsilon \in (0,1)$. The proof is divided in 3 steps.
  
  \emph{Step~1.} In this step, we construct $\eta_0 > 0$, depending on $\epsilon$, such that for all $\eta \leq \eta_0$, there exists $n_0 \geq 2$ which depends on $\eta$ such that, for all $n \geq n_0$, for all $\tbx \in \Mdn$, if
  \begin{equation}\label{eq:expcontrol:normel}
    \frac{1}{n}\sum_{i=1}^n |\tx_i|^\ell \geq 1,
  \end{equation}
  then
  \begin{equation}\label{eq:expcontrol:epsilon}
    (1-q)\hat{V}^\eta_n(\tbx)+\intene_n(\tbx) \geq \hat{V}^{\eta(1-\epsilon)}_n(\tbx)+(1-\epsilon)\intene_n(\tbx).
  \end{equation}
  We first rewrite~\eqref{eq:expcontrol:epsilon} under the equivalent formulation
  \begin{equation*}
    \epsilon\intene_n(\tbx) \geq (q-1)\hat{V}^\eta_n(\tbx) + \hat{V}^{\eta(1-\epsilon)}_n(\tbx).
  \end{equation*}
  On the one hand, the upper bound of Lemma~\ref{lem:hatVeta-bounds} yields, for all $\tbx \in \Mdn$,
  \begin{equation*}
    (q-1)\hat{V}^\eta_n(\tbx) + \hat{V}^{\eta(1-\epsilon)}_n(\tbx) \leq 2^{\ell-1}\eta(q-\epsilon)\frac{1}{n}\sum_{i=1}^n|\tx_i|^\ell + \alpha(n,\eta,\epsilon),
  \end{equation*}
  with
  \begin{equation*}
    \alpha(n,\eta,\epsilon) := (q-1)\frac{\sigma^2}{2n}\log\frac{(2^{\ell-1}n)^{d/\ell}}{z^\eta} + \frac{\sigma^2}{2n}\log\frac{(2^{\ell-1}n)^{d/\ell}}{z^{\eta(1-\epsilon)}};
  \end{equation*}
  on the other hand, Assumption~\eqref{ass:GC} yields, for all $\tbx \in \Mdn$,
  \begin{equation*}
    \intene_n(\tbx) \geq \frac{\kappa_\ell}{n}\sum_{i=1}^n|\tx_i|^\ell.
  \end{equation*}
  We deduce that~\eqref{eq:expcontrol:epsilon} holds as soon as
  \begin{equation}\label{eq:expcontrol:epsilon:2}
    \left(\epsilon\kappa_\ell - 2^{\ell-1}\eta(q-\epsilon)\right)\frac{1}{n}\sum_{i=1}^n|\tx_i|^\ell \geq \alpha(n,\eta,\epsilon).
  \end{equation}
  With the latter condition at hand, let us define
  \begin{equation*}
    \eta_0 = \frac{\epsilon\kappa_\ell}{2^\ell(q-\epsilon)}
  \end{equation*}
  and notice that, for all $\eta \leq \eta_0$,
  \begin{equation*}
    \lim_{n \to +\infty} \alpha(n,\eta,\epsilon) = 0
  \end{equation*}
  so that there exists $n_0 \geq 2$, depending on $\eta$, such that, for all $n \geq n_0$, 
  \begin{equation*}
    \alpha(n,\eta,\epsilon) \leq \frac{\epsilon\kappa_\ell}{2},
  \end{equation*}
  and therefore
  \begin{equation*}
    \left(\epsilon\kappa_\ell - 2^{\ell-1}\eta(q-\epsilon)\right) \geq \frac{\epsilon\kappa_\ell}{2} \geq \alpha(n,\eta,\epsilon).
  \end{equation*}
  As a conclusion, for all $n \geq n_0$, if $\tbx \in \Mdn$ satisfies~\eqref{eq:expcontrol:normel} then~\eqref{eq:expcontrol:epsilon:2} holds, which leads to~\eqref{eq:expcontrol:epsilon}.
  
  \emph{Step~2.} Let us fix $\eta \leq \eta_0$ and $n \geq n_0$, where $\eta_0$ and $n_0$ are given by Step~1. In this step, we give an upper bound on 
  \begin{equation*}
    I^\eta_n := \int_{\tbx \in \Mdn} \exp\left(\frac{2nq}{\sigma^2}\hat{V}^\eta_n(\tbx)\right)\dd\tP^\eta_n(\tbx)
  \end{equation*}
  by studying this integral separately on the domains
  \begin{equation*}
    D^\ell_{n,d} := \left\{\tbx \in \Mdn : \frac{1}{n}\sum_{i=1}^n |\tx_i|^\ell < 1\right\}
  \end{equation*}
  and on its complement. By the upper bound of Lemma~\ref{lem:hatVeta-bounds},
  \begin{equation*}
    \int_{\tbx \in D^\ell_{n,d}} \exp\left(\frac{2nq}{\sigma^2}\hat{V}^\eta_n(\tbx)\right)\dd\tP^\eta_n(\tbx) \leq \left(\frac{(2^{\ell-1}n)^{d/\ell}}{z^\eta}\right)^q\exp\left(\frac{2^\ell nq\eta}{\sigma^2}\right).
  \end{equation*}
  On the other hand, using Lemma~\ref{lem:tPPeta-alt} and Step~1 we obtain the chain of inequalities
  \begin{equation}\label{eq:expcontrol:Dc}
    \begin{aligned}
      & \int_{\tbx \in \Mdn \setminus D^\ell_{n,d}} \exp\left(\frac{2nq}{\sigma^2}\hat{V}^\eta_n(\tbx)\right)\dd\tP^\eta_n(\tbx)\\
      & \qquad = \frac{1}{\tZ^\eta_n}\int_{\tbx \in \Mdn  \setminus D^\ell_{n,d}} \exp\left(-\frac{2n}{\sigma^2}\left((1-q)\hat{V}^\eta_n(\tbx)+\intene_n(\tbx)\right)\right)\dd\tbx\\
      & \qquad \leq \frac{1}{\tZ^\eta_n}\int_{\tbx \in \Mdn  \setminus D^\ell_{n,d}} \exp\left(-\frac{2n}{\sigma^2}\left(\hat{V}^{\eta(1-\epsilon)}_n(\tbx)+(1-\epsilon)\intene_n(\tbx)\right)\right)\dd\tbx\\
      & \qquad \leq \frac{1}{\tZ^\eta_n}\int_{\tbx \in \Mdn} \exp\left(-\frac{2n}{\sigma^2}\left(\hat{V}^{\eta(1-\epsilon)}_n(\tbx)+(1-\epsilon)\intene_n(\tbx)\right)\right)\dd\tbx.
    \end{aligned}
  \end{equation}
  By Assumption~\eqref{ass:SH}, $(1-\epsilon)\intene_n(\tbx) \geq \intene_n((1-\epsilon)\tbx)$. We now derive a similar bound for $\hat{V}^{\eta(1-\epsilon)}_n(\tbx)$. The definition~\eqref{eq:hatVeta} yields
  \begin{equation*}
    \begin{aligned}
      & \hat{V}^\eta_n((1-\epsilon)\tbx) = -\frac{\sigma^2}{2n} \log \int_{\zeta\in\R^d} \exp\left(-\frac{2\eta}{\sigma^2}\sum_{i=1}^n |(1-\epsilon)\tx_i+\zeta|^\ell\right)\dd\zeta\\
      & \qquad \leq -\frac{\sigma^2 d}{2n}\log(1-\epsilon) -\frac{\sigma^2}{2n} \log \int_{\xi\in\R^d} \exp\left(-\frac{2\eta(1-\epsilon)}{\sigma^2}\sum_{i=1}^n|\tx_i+\xi|^\ell\right)\dd\xi\\
      & \qquad = -\frac{\sigma^2 d}{2n}\log(1-\epsilon) + \hat{V}^{\eta(1-\epsilon)}_n(\tbx),
    \end{aligned}
  \end{equation*}
  where we have performed the change of variable $\zeta=(1-\epsilon)\xi$ and used the fact that $(1-\epsilon)^\ell \leq 1-\epsilon$. Thus,
  \begin{equation*}
    \begin{aligned}
      & \int_{\tbx \in \Mdn} \exp\left(-\frac{2n}{\sigma^2}\left(\hat{V}^{\eta(1-\epsilon)}_n(\tbx)+(1-\epsilon)\intene_n(\tbx)\right)\right)\dd\tbx\\
      & \leq \int_{\tbx \in \Mdn} \exp\left(-\frac{2n}{\sigma^2}\left(\hat{V}^{\eta}_n((1-\epsilon)\tbx)+\intene_n((1-\epsilon)\tbx)\right) -d\log(1-\epsilon)\right)\dd\tbx\\
      & \leq \frac{1}{(1-\epsilon)^{dn}}\int_{\tby \in \Mdn} \exp\left(-\frac{2n}{\sigma^2}\left(\hat{V}^{\eta}_n(\tby)+\intene_n(\tby)\right)\right)\dd\tby = \frac{\tZ^\eta_n}{(1-\epsilon)^{dn}},
    \end{aligned}
  \end{equation*}
  thanks to the change of variable $\tby = (1-\epsilon)\tbx$. Injecting this inequality at the end of~\eqref{eq:expcontrol:Dc}, we obtain
  \begin{equation*}
    \int_{\tbx \in \Mdn \setminus D^\ell_{n,d}} \exp\left(\frac{2nq}{\sigma^2}\hat{V}^\eta_n(\tbx)\right)\dd\tP^\eta_n(\tbx) \leq \frac{1}{(1-\epsilon)^{dn}},
  \end{equation*}
  so that we may conclude this step by stating that
  \begin{equation*}
    I^\eta_n \leq \left(\frac{(2^{\ell-1}n)^{d/\ell}}{z^\eta}\right)^q\exp\left(\frac{2^\ell nq\eta}{\sigma^2}\right) + \frac{1}{(1-\epsilon)^{dn}}.
  \end{equation*}
  
  \emph{Step~3.} We complete the proof by studying the asymptotic behaviour of $I^\eta_n$. By Step~2 and the standard asymptotic subadditivity argument,
  \begin{equation*}
    \begin{aligned}
      & \limsup_{n \to +\infty} \frac{1}{n} \log I^\eta_n\\
      & \leq \limsup_{n \to +\infty} \frac{1}{n} \log\left(\left(\frac{(2^{\ell-1}n)^{d/\ell}}{z^\eta}\right)^q\exp\left(\frac{2^\ell nq\eta}{\sigma^2}\right) + \frac{1}{(1-\epsilon)^{dn}}\right)\\
      & \leq \limsup_{n \to +\infty} \frac{1}{n} \log\left(\left(\frac{(2^{\ell-1}n)^{d/\ell}}{z^\eta}\right)^q\exp\left(\frac{2^\ell nq\eta}{\sigma^2}\right)\right) + \limsup_{n \to +\infty} \frac{1}{n} \log \frac{1}{(1-\epsilon)^{dn}}\\
      & = \frac{2^\ell q\eta}{\sigma^2} - d\log(1-\epsilon),
    \end{aligned}
  \end{equation*}
  from which we then deduce that
  \begin{equation*}
    \limsup_{\eta \dto 0} \limsup_{n \to +\infty} \frac{1}{n} \log I^\eta_n \leq - d\log(1-\epsilon).
  \end{equation*}
  We may now complete the proof of~\eqref{eq:expcontrol} by letting $\epsilon$ vanish.
\end{proof}

\subsection{Large deviation upper and lower bounds}\label{ss:bounds} In this subsection, we complete the proof of Theorems~\ref{theo:Wass} and~\ref{theo:quot} by addressing the large deviation upper and lower bounds.

\begin{lem}[Large deviation upper bound]\label{lem:ub}
  Let $\IntEne : \Ps(\R^d) \to [0,+\infty]$ be an energy functional satisfying Assumptions~\eqref{ass:TI}, \eqref{ass:F}, \eqref{ass:LSC}, \eqref{ass:GC}, \eqref{ass:SH} and~\eqref{ass:CC}.
  \begin{enumerate}[label=(\roman*),ref=\roman*]
    \item For all closed sets $\bar{B}$ of $\bPs(\R^d)$,
    \begin{equation*}
      \limsup_{n \to +\infty} \frac{1}{n} \log \bPP_n(\bar{B}) \leq -\inf_{\bmu \in \bar{B}} \bRate[\bmu].
    \end{equation*}
    \item If the index $\ell \geq 1$ given by Assumption~\eqref{ass:GC} is such that $\ell > 1$, then for all $p \in [1,\ell)$, for all closed sets $\tilde{B}$ of $\tPs_p(\R^d)$,
    \begin{equation*}
      \limsup_{n \to +\infty} \frac{1}{n} \log \tPP_n(\tilde{B}) \leq -\inf_{\tmu \in \tilde{B}} \tRate[\tmu].
    \end{equation*}
  \end{enumerate}
\end{lem}
\begin{proof}
  We shall prove both statements at once. Let $B'$ refer to either $\bar{B} \subset \bPs(\R^d)$ or $\tilde{B} \subset \tPs_p(\R^d)$, $\PP'_n$ (resp. $\PP'^{\eta}_n$) refer to either $\bPP_n$ (resp. $\bPP^\eta_n$) or $\tPP_n$ (resp. $\tPP^\eta_n$), and so on. By Lemma~\ref{lem:PPetaPP}, for all $\eta > 0$, for all $n \geq 2$,
  \begin{equation}\label{eq:PPetaPP:pf}
    \begin{aligned}
      & \frac{1}{n}\log \PP'_n(B')\\
      & = \frac{1}{n}\log \frac{\tZ^\eta_n}{\tZ_n} + \frac{1}{n}\log \int_{\tbx \in \Mdn} \ind{\pi_n(\tbx) \in B''}\exp\left(\frac{2n}{\sigma^2}\hat{V}^\eta_n(\tbx)\right)\dd\tP^\eta_n(\tbx),
    \end{aligned}
  \end{equation}
  where $B''$ refers to either $\rho^{-1}(\bar{B})$ or $\tilde{B}$.
  
  By~\eqref{eq:tZeta} and the lower bound of Lemma~\ref{lem:hatVeta-bounds},
  \begin{equation*}
    \frac{\tZ^\eta_n}{\tZ_n} = \int_{\tbx \in \Mdn} \exp\left(-\frac{2n}{\sigma^2}\hat{V}^\eta_n(\tbx)\right)\dd\tP_n(\tbx) \leq z^\eta n^{-d/\ell},
  \end{equation*}
  whence
  \begin{equation*}
    \limsup_{n \to +\infty} \frac{1}{n}\log \frac{\tZ^\eta_n}{\tZ_n} \leq 0,
  \end{equation*}
  for any $\eta > 0$.

  Let us now fix $q,q' \in (1,+\infty)$ such that $1/q + 1/q'=1$. By Hölder's inequality, for all $\eta > 0$,
  \begin{equation*}
    \begin{aligned}
      & \log \int_{\tbx \in \Mdn} \ind{\pi_n(\tbx) \in B''}\exp\left(\frac{2n}{\sigma^2}\hat{V}^\eta_n(\tbx)\right)\dd\tP^\eta_n(\tbx)\\
      & \qquad \leq \frac{1}{q'} \log\int_{\tbx \in \Mdn} \ind{\pi_n(\tbx) \in B''}\dd\tP^\eta_n(\tbx)\\
      & \qquad \quad + \frac{1}{q} \log\int_{\tbx \in \Mdn} \exp\left(\frac{2nq}{\sigma^2}\hat{V}^\eta_n(\tbx)\right)\dd\tP^\eta_n(\tbx).
    \end{aligned}
  \end{equation*}
  By Lemma~\ref{lem:PPetaPP}, for all $\eta > 0$,
  \begin{equation*}
    \int_{\tbx \in \Mdn} \ind{\pi_n(\tbx) \in B''}\dd\tP^\eta_n(\tbx) = \PP'^\eta_n(B'),
  \end{equation*}
  and by Corollary~\ref{cor:Contraction},
  \begin{equation*}
    \limsup_{n \to +\infty} \frac{1}{n}\log \PP'^\eta_n(B') \leq -\inf_{\mu' \in B'} \Rate'^\eta[\mu'],
  \end{equation*}
  with $\Rate'^\eta$ referring to either $\bRate^\eta$ or $\tRate^\eta$. Using Lemma~\ref{lem:expcontrol}, we thus deduce that
  \begin{equation*}
    \begin{aligned}
      & \limsup_{n \to +\infty} \frac{1}{n}\log \PP'_n(B')\\
      & \qquad \leq \limsup_{\eta \dto 0} \limsup_{n \to +\infty} \frac{1}{n}\log \int_{\tbx \in \Mdn} \ind{\pi_n(\tbx) \in B''}\exp\left(\frac{2n}{\sigma^2}\hat{V}^\eta_n(\tbx)\right)\dd\tP^\eta_n(\tbx)\\
      & \qquad \leq \frac{1}{q'}\limsup_{\eta \dto 0} \left\{-\inf_{\mu' \in B'} \Rate'^\eta[\mu']\right\},
    \end{aligned}
  \end{equation*}
  from which we deduce that
  \begin{equation*}
    \limsup_{n \to +\infty} \frac{1}{n}\log \PP'_n(B') \leq \frac{1}{q'} \left\{-\inf_{\mu' \in B'} \Rate'[\mu']\right\},
  \end{equation*}
  thanks to Lemma~\ref{lem:cvtRate} stated below. Since $q'$ is arbitrarily close to $1$, the proof is completed.
\end{proof}

\begin{lem}[Large deviation lower bound]\label{lem:lb}
  Let $\IntEne : \Ps(\R^d) \to [0,+\infty]$ be an energy functional satisfying Assumptions~\eqref{ass:TI}, \eqref{ass:F}, \eqref{ass:LSC}, \eqref{ass:GC}, \eqref{ass:SH} and~\eqref{ass:CC}. 
  \begin{enumerate}[label=(\roman*),ref=\roman*]
    \item For all open sets $\bar{B}$ of $\bPs(\R^d)$,
    \begin{equation*}
      \liminf_{n \to +\infty} \frac{1}{n} \log \bPP_n(\bar{B}) \geq -\inf_{\bmu \in \bar{B}} \bRate[\bmu].
    \end{equation*}
    \item If the index $\ell \geq 1$ given by Assumption~\eqref{ass:GC} is such that $\ell > 1$, then for all $p \in [1,\ell)$, for all open sets $\tilde{B}$ of $\tPs_p(\R^d)$,
    \begin{equation*}
      \liminf_{n \to +\infty} \frac{1}{n} \log \tPP_n(\tilde{B}) \geq -\inf_{\tmu \in \tilde{B}} \tRate[\tmu].
    \end{equation*}
  \end{enumerate}
\end{lem}
\begin{proof}
  We shall prove both statements at once, and use the same shortcut notations as in the proof of Lemma~\ref{lem:ub}. Once again, we start from the fact that $\PP'_n(B')$ satisfies the identity~\eqref{eq:PPetaPP:pf}. Noting that
  \begin{equation*}
    \frac{\tZ^\eta_n}{\tZ_n} = \frac{1}{\displaystyle\int_{\tbx \in \Mdn} \exp\left(\frac{2n}{\sigma^2}\hat{V}^\eta_n(\tbx)\right)\dd\tP^\eta_n(\tbx)}
  \end{equation*}
  and then using Lemma~\ref{lem:expcontrol} with $q=1$, we first obtain
  \begin{equation*}
    \liminf_{\eta \dto 0} \liminf_{n \to +\infty} \frac{1}{n}\log \frac{\tZ^\eta_n}{\tZ_n} \geq 0.
  \end{equation*}
  We now combine the lower bound of Lemma~\ref{lem:hatVeta-bounds} with Lemma~\ref{lem:PPetaPP} to write
  \begin{equation*}
    \int_{\tbx \in \Mdn} \ind{\pi_n(\tbx) \in B''}\exp\left(\frac{2n}{\sigma^2}\hat{V}^\eta_n(\tbx)\right)\dd\tP^\eta_n(\tbx) \geq \frac{n^{d/\ell}}{z^\eta} \PP_n'^\eta(B'),
  \end{equation*} 
  from which we deduce that
  \begin{equation*}
    \liminf_{n \to +\infty} \frac{1}{n}\log\PP'_n(B') \geq \liminf_{\eta \dto 0} \left\{-\inf_{\mu' \in B'} \Rate'^\eta[\mu']\right\}
  \end{equation*}
  thanks to Corollary~\ref{cor:Contraction}. The conclusion follows from the application of Lemma~\ref{lem:cvtRate}, which is stated below.
\end{proof}

\begin{lem}[Convergence of rate functions]\label{lem:cvtRate}
  Under the assumptions of either Theorem~\ref{theo:quot} or Theorem~\ref{theo:Wass}, let $\Rate'$ (resp. $\Rate'^\eta$, for $\eta > 0$) refer to either $\bRate$ (resp. $\bRate^\eta$) or $\tRate$ (resp. $\tRate^\eta$). Then for any subset $B'$ of either $\bPs(\R^d)$ or $\tPs_p(\R^d)$,
  \begin{equation*}
    \lim_{\eta \dto 0} \inf_{\mu' \in B'} \Rate'^\eta[\mu'] = \inf_{\mu' \in B'} \Rate'[\mu'].
  \end{equation*}
\end{lem}
\begin{proof}
  The functions $\Rate'^\eta$ and $\Rate'$ write
  \begin{equation*}
    \Rate'^\eta[\mu'] = \FreEne'[\mu'] + \frac{2\eta}{\sigma^2} \vartheta'[\mu'] - \FreEne^\eta_\star, \qquad \Rate'[\mu'] = \FreEne'[\mu'] - \FreEne_\star,
  \end{equation*}
  with obvious notations for $\FreEne'$ and $\vartheta'$, and
  \begin{equation*}
    \FreEne^\eta_\star = \inf_{\mu'} \left\{\FreEne'[\mu'] + \frac{2\eta}{\sigma^2} \vartheta'[\mu']\right\}, \qquad \FreEne_\star = \inf_{\mu'} \FreEne'[\mu'].
  \end{equation*}
  Thus, it is sufficient to prove that
  \begin{equation*}
    \lim_{\eta \dto 0} \inf_{\mu' \in B'} \left\{\FreEne'[\mu'] + \frac{2\eta}{\sigma^2} \vartheta'[\mu']\right\} = \inf_{\mu' \in B'} \FreEne'[\mu'].
  \end{equation*}
  The fact that $\vartheta'[\mu']\geq 0$ immediately yields
  \begin{equation*}
    \liminf_{\eta \dto 0} \inf_{\mu' \in B'} \left\{\FreEne'[\mu'] + \frac{2\eta}{\sigma^2} \vartheta'[\mu']\right\} \geq \inf_{\mu' \in B'} \FreEne'[\mu'].
  \end{equation*}
  Furthermore, for any $\nu' \in B'$, 
  \begin{equation*}
    \inf_{\mu' \in B'} \left\{\FreEne'[\mu'] + \frac{2\eta}{\sigma^2} \vartheta'[\mu']\right\} \leq \FreEne'[\nu'] + \frac{2\eta}{\sigma^2} \vartheta'[\nu'],
  \end{equation*}
  and letting $\eta$ vanish in both sides of the inequality yields
  \begin{equation*}
    \limsup_{\eta \dto 0} \inf_{\mu' \in B'} \left\{\FreEne'[\mu'] + \frac{2\eta}{\sigma^2} \vartheta'[\mu']\right\} \leq \FreEne'[\nu'],
  \end{equation*}
  so that taking the infimum of the right-hand side over $\nu'$ yields
  \begin{equation*}
    \limsup_{\eta \dto 0} \inf_{\mu' \in B'} \left\{\FreEne'[\mu'] + \frac{2\eta}{\sigma^2} \vartheta'[\mu']\right\} \leq \inf_{\nu' \in B'} \FreEne'[\nu'],
  \end{equation*}
  which completes the proof.
\end{proof}

\section{Application to McKean-Vlasov and rank-based models}\label{s:appl}

\subsection{\texorpdfstring{\MV}{MV}-model}\label{ss:MV} This subsection presents the proof of Corollary~\ref{cor:MV}. We first assume that, in the decomposition~\eqref{eq:WsWf}, $W^\flat \equiv 0$.

\begin{lem}[Case $W^\flat \equiv 0$]\label{lem:MV-Ws}
  Let $W^\sharp : \R^d \to [0,+\infty)$ be an interaction potential satisfying Assumption~\eqref{ass:MVs}. Then the associated energy functional $\IntEne^\sharp$ defined by~\eqref{eq:MVIntEne} with $W = W^\sharp$ satisfies Assumptions~\eqref{ass:TI}, \eqref{ass:F}, \eqref{ass:LSC}, \eqref{ass:GC}, \eqref{ass:SH} and~\eqref{ass:CC}; besides, Assumption~\eqref{ass:GC} holds with the index $\ell \geq 1$ given by Assumption~\eqref{ass:MVs}.
\end{lem}
\begin{proof}
  Assumptions~\eqref{ass:TI} and~\eqref{ass:F} are straightforward. The continuity of the mapping $\mu \mapsto \mu \otimes \mu$ on $\Ps(\R^d)$, combined with the fact that, by Assumtion~\eqref{ass:MVs}, $W^\sharp$ is nonnegative and lower semicontinuous, and Fatou's Lemma, yield Assumption~\eqref{ass:LSC}.
  
  Let $\ell \geq 1$ be given by Assumption~\eqref{ass:MVs}. By~\eqref{eq:MVIntEne}, for all $\mu \in \Ps(\R^d)$, 
  \begin{equation*}
    \IntEne^\sharp[\mu] \geq \kappa_\ell\iint_{x,y \in \R^d} |x-y|^\ell \dd\mu(x)\dd\mu(y),
  \end{equation*}
  which, by the Fubini-Tonelli Theorem, implies that $\IntEne^\sharp[\mu]=+\infty$ if $\mu \not\in \Ps_\ell(\R^d)$. On the other hand, if $\tmu \in \tPs_\ell(\R^d)$, then by Jensen's Inequality,
  \begin{equation*}
    \begin{aligned}
      \int_{x \in \R^d} |x|^\ell \dd\tmu(x) & = \int_{x \in \R^d} \left|x-\int_{y \in \R^d} y \dd\tmu(y)\right|^\ell \dd\tmu(x)\\
      & = \int_{x \in \R^d} \left|\int_{y \in \R^d} (x-y) \dd\tmu(y)\right|^\ell \dd\tmu(x)\\
      & \leq \iint_{x,y \in \R^d} \left|x-y\right|^\ell \dd\tmu(x)\dd\tmu(y),
    \end{aligned}
  \end{equation*}
  so that
  \begin{equation*}
    \IntEne^\sharp[\tmu] \geq \kappa_\ell \int_{x \in \R^d} |x|^\ell \dd\tmu(x),
  \end{equation*}
  and $\IntEne^\sharp$ satisfies Assumption~\eqref{ass:GC}.
  
  Assumption~\eqref{ass:SH} is a straightforward consequence of Assumption~\eqref{ass:MVs}.
  
  We finally let $\mu \in \Ps(\R)$ and take a sequence of independent random variables $(Y_n)_{n \geq 1}$ on some probability space $(\Omega,\mathcal{A},\Pr)$ with identical distribution $\mu$. For all $n \geq 2$, 
  \begin{equation*}
    \Exp[\intene^\sharp_n(Y_1, \ldots, Y_n)] = \frac{1}{2n^2} \sum_{\substack{i,j=1\\ i\not=j}}^n \Exp[W^\sharp(Y_1-Y_2)] = \frac{n-1}{n} \IntEne^\sharp[\mu],
  \end{equation*}
  which leads to Assumption~\eqref{ass:CC} and completes the proof.
\end{proof}

We now address the general case $W=W^\sharp+W^\flat$, with $W^\flat \not\equiv 0$. We decompose the energy functional $\IntEne$, defined by~\eqref{eq:MVIntEne}, as
\begin{equation*}
  \IntEne = \IntEne^\sharp+\IntEne^\flat,
\end{equation*}
with obvious definitions for $\IntEne^\sharp$ and $\IntEne^\flat$. By Lemma~\ref{lem:MV-Ws} and Theorems~\ref{theo:quot} and~\ref{theo:Wass}, the sequences $\bPP^\sharp_n$ and $\tPP^\sharp_n$ associated with $\IntEne^\sharp$ satisfy the large deviation principles of Corollary~\ref{cor:MV}, with respective rate functions denoted by $\bRate^\sharp$ and $\tRate^\sharp$. On the other hand, 
\begin{equation*}
  \frac{\dd\bPP_n}{\dd\bPP^\sharp_n}[\bmu] = \frac{\tZ^\sharp_n}{\tZ_n}\exp\left(-\frac{2n}{\sigma^2}\bIntEne^\flat[\bmu]\right), \qquad \frac{\dd\tPP_n}{\dd\tPP^\sharp_n}[\tmu] = \frac{\tZ^\sharp_n}{\tZ_n}\exp\left(-\frac{2n}{\sigma^2}\IntEne^\flat[\tmu]\right),
\end{equation*}
with an obvious definition for $\tZ^\sharp_n$.

If $\ell = 1$, then by Assumption~\eqref{ass:MVf}, $\bIntEne^\flat$ is a bounded and continuous functional on $\bPs(\R^d)$, so that the application of the Laplace-Varadhan Lemma~\cite[Theorem~II.7.2, p.~52]{Ell85} is straightforward and yields the first part of Corollary~\ref{cor:MV}.

\begin{rk}[On the Laplace-Varadhan Lemma]
  The statement of the Laplace-Varadhan Lemma in~\cite[Theorem~II.7.2, p.~52]{Ell85} requires the state space to be Polish, which is not proved for $\bPs(\R^d)$ in the present article. However, a careful examination of the proof of this theorem shows that this assumption is in fact not necessary. More generally, we refer to~\cite[Section~4.3]{DemZei10} for an exposition of Varadhan's Lemma and various developments on regular (and in particular metric) topological spaces, which are not necessarily Polish.
\end{rk}

Let us now assume that $\ell>1$, and fix $p \in [\max(1,\ell'),\ell)$, where $\ell' \in [0,\ell)$ is given by Assumption~\eqref{ass:MVf}. The functional $\IntEne^\flat$ is continuous on $\tPs_p(\R^d)$, but not necessarily bounded, so that following~\cite[Theorem~II.7.2, p.~52]{Ell85} and~\cite[Lemma~4.3.8, p.~138]{DemZei10}, we shall check the exponential moment condition
\begin{equation}\label{eq:unifexpint}
  \limsup_{n \to +\infty} \frac{1}{n} \log \int_{\tmu \in \tPs_p(\R^d)} \exp\left(-\frac{2n\gamma}{\sigma^2}\IntEne^\flat[\tmu]\right) \dd\tPP^\sharp_n[\tmu] < +\infty,
\end{equation}
for some $\gamma>1$ --- in fact, since any multiple of $\IntEne^\flat$ also satisfies Assumption~\eqref{ass:MVf}, this condition should hold for any $\gamma \in \R$. 

Taking~\eqref{eq:unifexpint} for granted, the Laplace-Varadhan Lemma~\cite[Theorem~II.7.2, p.~52]{Ell85} allows to transfer the large deviation principle from $\tPP^\sharp_n$ to $\tPP_n$ on $\tPs_p(\R^d)$, for any $p \in [\max(1,\ell'),\ell)$. This result is then extended on $\tPs_p(\R^d)$, for any $p \in [1,\ell)$, and to $\bPs(\R^d)$, by the use of the Contraction Principle~\cite[Theorem~4.2.1, p.~126]{DemZei10}, which completes the proof of Corollary~\ref{cor:MV}.

\begin{proof}[Proof of~\eqref{eq:unifexpint}]
  The argument is similar to the proof of Lemma~\ref{lem:expcontrol}. Let us fix $\gamma > 1$, and rewrite
  \begin{equation}\label{eq:unifexpint:pf1}
    \begin{aligned}
      & \int_{\tmu \in \tPs_p(\R^d)} \exp\left(-\frac{2n\gamma}{\sigma^2}\IntEne^\flat[\tmu]\right) \dd\tPP^\sharp_n[\tmu]\\
      & \qquad = \frac{\displaystyle \int_{\tbx \in \Mdn} \exp\left(-\frac{2n}{\sigma^2}\left(\gamma\intene^\flat_n(\tbx)+\intene^\sharp_n(\tbx)\right)\right)\dd\tbx}{\displaystyle \int_{\tbx \in \Mdn} \exp\left(-\frac{2n}{\sigma^2}\intene^\sharp_n(\tbx)\right)\dd\tbx}.
    \end{aligned}
  \end{equation}
  Assumption~\eqref{ass:MVf} and Jensen's Inequality imply that there exists $C \geq 0$ such that, for all $\tbx \in \Mdn$,
  \begin{equation*}
    \left|\gamma\intene^\flat_n(\tbx)\right| \leq \frac{C\gamma}{2}\left(1 + \frac{2^{\ell'}}{n}\sum_{i=1}^n |\tx_i|^{\ell'}\right).
  \end{equation*}
  By Hölder's Inequality and Assumption~\eqref{ass:MVs},
  \begin{equation*}
    \frac{1}{n}\sum_{i=1}^n |\tx_i|^{\ell'} \leq \left(\frac{1}{n}\sum_{i=1}^n |\tx_i|^\ell\right)^{\ell'/\ell} \leq \left(\frac{1}{\kappa_\ell}\intene_n^\sharp(\tbx)\right)^{\ell'/\ell},
  \end{equation*}
  so that
  \begin{equation*}
    \gamma\intene^\flat_n(\tbx) \geq -\frac{C\gamma}{2}\left(1 + \frac{2^{\ell'}}{\kappa_\ell^{\ell'/\ell}}\left(\intene_n^\sharp(\tbx)\right)^{\ell'/\ell}\right),
  \end{equation*}
  and for any $\epsilon \in (0,1)$, there exists $L \geq 0$ such that, for all $n \geq 2$, for all $\tbx \in \Mdn$, the condition
  \begin{equation*}
    \intene_n^\sharp(\tbx) \geq L
  \end{equation*}
  implies that 
  \begin{equation*}
    \gamma\intene^\flat_n(\tbx) + \intene_n^\sharp(\tbx) \geq (1-\epsilon) \intene_n^\sharp(\tbx).
  \end{equation*}
  Studying the integral in the numerator of the right-hand side of~\eqref{eq:unifexpint:pf1} separately on the domains $\{\intene_n^\sharp(\tbx) < L\}$ and on its complement, we get the bound
  \begin{equation*}
    \begin{aligned}
      & \int_{\tmu \in \tPs_p(\R^d)} \exp\left(-\frac{2n\gamma}{\sigma^2}\IntEne^\flat[\tmu]\right) \dd\tPP^\sharp_n[\tmu]\\
      & \leq \exp\left(\frac{nC\gamma}{\sigma^2}\left(1 + \frac{2^{\ell'}L^{\ell'/\ell}}{\kappa_\ell^{\ell'/\ell}}\right)\right) + \frac{\displaystyle \int_{\tbx \in \Mdn} \exp\left(-\frac{2n}{\sigma^2}(1-\epsilon)\intene^\sharp_n(\tbx)\right)\dd\tbx}{\displaystyle \int_{\tbx \in \Mdn} \exp\left(-\frac{2n}{\sigma^2}\intene^\sharp_n(\tbx)\right)\dd\tbx},
    \end{aligned}
  \end{equation*}
  and the same change of variable as in the proof of Lemma~\ref{lem:expcontrol} allows to complete the proof of~\eqref{eq:unifexpint}.
\end{proof}
\subsection{\texorpdfstring{\RB}{RB}-model}\label{ss:RB} The next lemma allows to deduce Corollary~\ref{cor:RB} from a straightforward application of Theorem~\ref{theo:quot}.

\begin{lem}[Assumptions of Theorem~\ref{theo:quot} for the $\RB$-model]
  If the flux function $B$ satisfies the assumptions of Corollary~\ref{cor:RB}, then the associated energy functional $\IntEne$ defined by~\eqref{eq:RBIntEne} satisfies Assumptions~\eqref{ass:TI}, \eqref{ass:F}, \eqref{ass:LSC}, \eqref{ass:GC} with $\ell=1$, \eqref{ass:SH} and~\eqref{ass:CC}.
\end{lem} 
\begin{proof}
  Assumptions~\eqref{ass:TI} and~\eqref{ass:F} are straightforward.

  To check Assumption~\eqref{ass:LSC}, we recall that the weak convergence of probability measures implies the convergence $\dd x$-almost everywhere of their cumulative distribution functions, so that the lower semicontinuity of $\IntEne$ follows from Fatou's Lemma and the fact that $B$ is continuous and nonnegative.

  Let us now define
  \begin{equation*}
    \kappa := \inf_{u \in (0,1)} \frac{B(u)}{2u(1-u)}.
  \end{equation*}
  The combination of the conditions~\eqref{eq:RBti},~\eqref{eq:Oleinik} and~\eqref{eq:Lax} implies that $\kappa \in (0,+\infty)$. As a consequence, for all $\mu \in \Ps(\R)$,
  \begin{equation*}
    \IntEne[\mu] \geq 2\kappa \int_{x \in \R} F_\mu(x)(1-F_\mu(x))\dd x = \kappa \iint_{x,y \in \R} |x-y|\dd\mu(x)\dd\mu(y),
  \end{equation*}
  which by Lemma~\ref{lem:MV-Ws} implies Assumption~\eqref{ass:GC} with $\ell=1$.

  Assumption~\eqref{ass:SH} follows from the remark that
  \begin{equation*}
    \intene_n(\bx) = - \frac{1}{n}\sum_{k=1}^n b_n(k) x_{(k)}, \qquad x_{(1)} \leq \cdots \leq x_{(n)},
  \end{equation*}
  so that $(1-\epsilon)\intene_n(\bx) = \intene_n((1-\epsilon)\bx)$.

  We finally let $\mu \in \Ps(\R)$ and take a sequence of independent random variables $(Y_n)_{n \geq 1}$ on some probability space $(\Omega,\mathcal{A},\Pr)$ with identical distribution $\mu$. If $\mu \not\in \Ps_1(\R)$, then by Remark~\ref{rk:CCLSC}, 
  \begin{equation*}
    \lim_{n \to +\infty} \Exp[\intene_n(Y_1, \ldots, Y_n)] = \IntEne[\mu] = +\infty.
  \end{equation*}
  On the contrary, if $\mu \in \Ps_1(\R)$, let us write
  \begin{equation*}
    \begin{aligned}
      \left|\Exp\left[\intene_n(Y_1, \ldots, Y_n)\right] - \IntEne[\mu]\right| & = \left|\Exp\left[\IntEne[\pi_n] - \IntEne[\mu]\right]\right| \leq \Exp\left[\left|\IntEne[\pi_n] - \IntEne[\mu]\right|\right],
    \end{aligned}
  \end{equation*}
  where $\pi_n$ is a short notation for $\pi_n(Y_1, \ldots, Y_n)$. Denoting $C = \sup_{u \in [0,1]} |B'(u)|$, we get
  \begin{equation*}
    \left|\IntEne[\pi_n] - \IntEne[\mu]\right| \leq C \int_{x \in \R} |F_{\pi_n}(x)-F_\mu(x)|\dd x = C \Ws_1(\pi_n,\mu),
  \end{equation*}
  where $\Ws_1$ is the Wasserstein distance of order $1$. That this distance coincides with the $L^1$ distance of cumulative distribution functions is a specific feature of probability measures on the real line, see~\cite[Theorem~2.9, p.~16]{BobLed14}. By~\cite[Theorem~2.14, p.~20]{BobLed14}, $\Exp[\Ws_1(\pi_n,\mu)]$ converges to $0$ when $n$ grows to infinity, which shows that $\IntEne$ satisfies Assumption~\eqref{ass:CC}.
\end{proof}

\subsection{Application to the study of atypical capital distribution}\label{ss:SPT} It was proved in~\cite{Rey15} that under the assumptions of Corollary~\ref{cor:RB}, $\tPP_n$ converges weakly, on $\tPs_p(\R)$ for any $p \geq 1$, to the Dirac mass $\delta_{\tmu_\infty}$, where $\tmu_\infty$ is the unique centered stationary measure of the nonlinear diffusion process describing the mean-field limit of~\eqref{eq:RBsde} --- we refer to~\cite{Shk12, JouRey13, DemShkVarZei16} for associated propagation of chaos results in the space of sample-paths. This measure satisfies the stationary nonlinear Fokker-Planck equation
\begin{equation}\label{eq:EulLag}
  0 = \frac{\sigma^2}{2}\partial_{xx} \tmu_\infty - \partial_x\left(b(F_{\tmu_\infty})\tmu_\infty\right), \qquad b(u) := B'(u),
\end{equation}
which implies that it possesses a density $\tp_\infty$ with respect to the Lebesgue measure on $\R$, which solves the fixed point relation
\begin{equation}\label{eq:tpinfty}
  \begin{aligned}
    & \tp_\infty(x) = \frac{1}{\tz_\infty}\exp\left(\frac{2}{\sigma^2}\int_{y=0}^x b(F_{\tmu_\infty}(y))\dd y\right),\\
    & \tz_\infty = \int_{x \in \R} \exp\left(\frac{2}{\sigma^2}\int_{y=0}^x b(F_{\tmu_\infty}(y))\dd y\right)\dd x.
  \end{aligned}
\end{equation}
As a consequence, if we let $(\tX_1, \ldots, \tX_n)$ be a random vector with distribution $\tP_n$, and $\tX_{\infty,1}, \ldots, \tX_{\infty,n}$ be independent random variables with identical distribution $\tmu_\infty$, then $\pi_n(\tX_1, \ldots, \tX_n)$ and $\pi_n(\tX_{\infty,1}, \ldots, \tX_{\infty,n})$ satisfy the same weak law of large numbers, and converge to $\tmu_\infty$. However, the large deviations of these random empirical measures are respectively described by Corollary~\ref{cor:RB} (in the orbit space $\bPs(\R)$), and by Sanov's Theorem. We first examine the difference between the associated rate functions, and then detail an application of this result to the estimation of the probability of atypical capital distribution in the context of Stochastic Portfolio Theory.

\subsubsection{Difference between rate functions} With the notations introduced above, let us define
\begin{equation*}
  \PP_{\infty,n} := \tmu_\infty^{\otimes n} \circ \pi_n^{-1}, \qquad \bPP_{\infty,n} := \PP_{\infty,n} \circ \rho^{-1}.
\end{equation*}
By Sanov's Theorem and the Contraction Principle, the sequence $\bPP_{\infty,n}$ satisfies a large deviation principle on $\bPs(\R)$, with good rate function
\begin{equation*}
  \bRate_\infty[\bmu] := \inf_{\mu \in \Ps(\R) : \rho(\mu)=\bmu} \RelEnt[\mu|\tmu_\infty].
\end{equation*}

\begin{lem}[Comparison of rate functions]\label{lem:comprate}
  Under the assumptions of Corollary~\ref{cor:RB}, we have, for all $\bmu \in \bPs(\R)$, for all $\mu \in \Ps(\R)$ such that $\rho(\mu)=\bmu$,
  \begin{equation}\label{eq:dlrate}
    \bRate[\bmu] = \RelEnt[\mu|\tmu_\infty] + \frac{2}{\sigma^2}\int_{x \in \R} \Gamma\left(F_\mu(x),F_{\tmu_\infty}(x)\right)\dd x,
  \end{equation}
  where
  \begin{equation*}
    \Gamma(u,v) := B(u)-B(v)-b(v)(u-v).
  \end{equation*}
  As a consequence, if $B$ is concave, then for all $\bmu \in \bPs(\R)$,
  \begin{equation}\label{eq:comrate}
    \bRate[\bmu] \leq \bRate_\infty[\bmu].
  \end{equation}
\end{lem}
\begin{proof}
  As a preliminary remark, we observe that the convergence of $\tPP_n$ to $\delta_{\tmu_\infty}$ implies that $\bPP_n$ converges weakly to $\delta_{\bmu_\infty}$, with $\bmu_\infty := \rho(\tmu_\infty)$, on $\bPs(\R)$. Combining this weak law of large numbers with Corollary~\ref{cor:RB}, we get that $\bmu_\infty$ is the unique zero of $\bRate$, and therefore that
\begin{equation*}
  \FreEne_\star = \bFreEne[\bmu_\infty] = \FreEne[\tmu_\infty].
\end{equation*}
Notice that~\eqref{eq:EulLag} is the optimality condition associated with the definition of $\FreEne_\star$.

  We now let $\bmu \in \bPs(\R)$ and $\mu \in \Ps(\R)$ be such that $\rho(\mu)=\bmu$. If $\mu \not\in \Ps_1(\R)$, then by Assumption~\eqref{ass:GC}, $\FreEne[\mu]=+\infty$ so that $\bRate[\bmu]=+\infty$; besides, it is known that $\tmu_\infty$ has exponential tails~\cite{JouRey13,JouRey15} so that $\bRate_\infty[\bmu]=+\infty$. Likewise, if $\mu$ is not absolutely continuous with respect to the Lebesgue measure on $\R$, then both $\bRate[\bmu]$ and $\bRate_\infty[\bmu]$ are infinite.
  
  We now assume that $\mu \in \Ps_1(\R)$ and has a density $p$ with respect to the Lebesgue measure, and write
  \begin{equation*}
    \begin{aligned}
      \bRate[\bmu] & = \bFreEne[\bmu]-\FreEne_\star\\
      & = \FreEne[\mu]-\FreEne[\tmu_\infty]\\
      & = \Ent[\mu] - \Ent[\tmu_\infty] + \frac{2}{\sigma^2}\int_{x \in \R} \left(B(F_\mu(x))-B(F_{\tmu_\infty}(x))\right)\dd x.
    \end{aligned}
  \end{equation*}
  Besides,
  \begin{equation*}
    \RelEnt[\mu|\tmu_\infty] = \Ent[\mu]-\Ent[\tmu_\infty] + \int_{x \in \R} \left(\tp_\infty(x)-p(x)\right)\log \tp_\infty(x) \dd x. 
  \end{equation*}
  By~\eqref{eq:tpinfty},
  \begin{equation*}
    \log \tp_\infty(x) = -\log \tz_\infty + \frac{2}{\sigma^2} \int_{y=0}^x b(F_{\tmu_\infty}(y))\dd y,
  \end{equation*}
  which after the use of Fubini's Theorem yields
  \begin{equation*}
    \begin{aligned}
      & \int_{x \in \R} \left(\tp_\infty(x)-p(x)\right)\log \tp_\infty(x) \dd x\\
      & \qquad = \frac{2}{\sigma^2}\int_{y \in \R} b(F_{\tmu_\infty}(y))\left(F_\mu(y)-F_{\tmu_\infty}(y)\right)\dd y,
    \end{aligned}
  \end{equation*}
  and leads to~\eqref{eq:dlrate}.
  
  If $B$ is concave, then $\Gamma(u,v) \leq 0$ for all $u,v \in [0,1]$, so that, for all $\bmu \in \bPs(\R)$,~\eqref{eq:dlrate} yields
  \begin{equation*}
    \bRate[\bmu] \leq \RelEnt[\mu|\tmu_\infty],
  \end{equation*}
  for all $\mu \in \Ps(\R)$ such that $\rho(\mu)=\bmu$. Taking the infimum over $\mu$ of the right-hand side results in~\eqref{eq:comrate} and completes the proof.
\end{proof}

\subsubsection{Capital distribution curves} In the framework of \emph{Stochastic Portfolio Theory}~\cite{Fer02, BanFerKar05}, systems of rank-based interacting diffusions of the form~\eqref{eq:RBsde} serve as first-order approximations of stable equity markets, in the sense that on a market with $n$ companies, the process $X_i(t)$ provides a good representation of the behaviour of the logarithmic capitalisation of the $i$-th company. Thus, the proportion of the total capital held by this company is given by its \emph{market weight}
\begin{equation*}
  \mu_i(t) := \frac{\exp(X_i(t))}{\sum_{j=1}^n \exp(X_j(t))}.
\end{equation*}
Using the reverse order statistics notation
\begin{equation*}
  \mu_{[1]}(t) \geq \cdots \geq \mu_{[n]}(t),
\end{equation*}
the \emph{capital distribution curve} is defined as the log-log plot of the mapping $m \mapsto \mu_{[m]}(t)$ and summarises in which manner the whole capital of a market is spread among the companies. 

Notice that the market weights are invariant by translation of $X_1(t), \ldots, X_n(t)$, so that the vector $(\mu_{[1]}(t), \ldots, \mu_{[n]}(t))$ (and therefore the associated capital distribution curve) only depends on $\rho(\pi_n(X_1(t), \ldots, X_n(t)))$. Besides, empirical studies (see for instance~\cite[Figure~5.1]{Fer02}) show that the capital distribution curves are remarkably stable over long times. These remarks suggest to study the statistical distribution of the vector $(\mu_{[1]}, \ldots, \mu_{[n]})$ under the probability measure $\bPP_n$~\cite{BanFerKar05, ChaPal10, JouRey15}. 

When $n$ grows to infinity, the law of large numbers for $\bPP_n$ prescribes a deterministic form for the (suitably rescaled) capital distribution curve, which was observed to fit empirical data in~\cite{JouRey15}. This defines a distribution of capital which we call \emph{typical}. If one wants to study the capital distribution without having to sample the high-dimensional vector $(\tX_1, \ldots, \tX_n)$ from the distribution $\tP_n$, then the discussion above shows that using independent random variables $\tX_{\infty,1}, \ldots, \tX_{\infty,n}$ identically distributed according to $\tmu_\infty$ as a surrogate model provides correct results concerning this typical behaviour. Such a surrogate model was for instance employed in~\cite{JouRey15} to evaluate the performance of diversity-weighted portfolios, and in~\cite[Section~3]{Bru16} to study hitting times and rank-rank correlations.

On the contrary, Lemma~\ref{lem:comprate} shows that the fluctuations of the capital distribution far from its typical behaviour, due to finite-size effects, and which are described by the large deviations of $\bPP_n$, are not correctly captured by the surrogate model in general. In short: both sequences $\bPP_n$ and $\bPP_{\infty,n}$ concentrate around $\bmu_\infty$, but their rate functions differ. On a more quantitative level, if the flux function $B$ is concave, then at the level of large deviations, the probability of an atypical distribution of the capital is \emph{always underestimated} by  $\bPP_{\infty,n}$ (the surrogate model) with respect to $\bPP_n$. In other words, the interaction between the stocks \emph{increases} the probability of an atypical capital distribution.

For similar works on the study of the fluctuations of mean-field rank-based interacting diffusions around, or far from, their typical behaviour, we refer to the respective works by Kolli and Shkolnikov~\cite{KolShk16}, and Dembo, Shkolnikov, Varadhan and Zeitouni~\cite{DemShkVarZei16}. We also mention that inequalities between rate functions for sequence of probability measures having the same law of large numbers, such as in Lemma~\ref{lem:comprate}, naturally provide comparisons between asymptotic variances in Monte-Carlo numerical methods. For more details in this direction, we refer to the work by Rey-Bellet and Spiliopoulos~\cite{ReySpi15} and the references therein. 

\appendix\section{Metrisability of the quotient topology on \texorpdfstring{$\bPs(\R^d)$}{bP}}\label{s:app}

By definition, the \emph{quotient topology} on $\bPs(\R^d)$ is the strongest topology making the orbit map $\rho$ continuous. The purpose of this appendix is to prove that this topology is metrisable, which is in general not the case for quotient topologies.

We first note that the definition of the quotient topology implies the following characterisation of open and closed sets.

\begin{lem}[Open and closed sets in $\bPs(\R^d)$]\label{lem:openclosedsets}
  A subset $\bar{A}$ of $\bPs(\R^d)$ is open (respectively closed) if and only if the set $\rho^{-1}(\bar{A})$ is open (respectively closed) in $\Ps(\R^d)$. 
\end{lem}

Our construction of a metric on $\bPs(\R^d)$ is based on the Prohorov metric on \texorpdfstring{$\Ps(\R^d)$}{P}, which following~\cite[Theorem~6.9, p.~74]{Bil99} can be defined by
\begin{equation}\label{eq:dP}
  \dP(\mu,\nu) := \inf_{\pi} \left\{\epsilon > 0 : \pi(\{(x,y) \in \R^d \times \R^d : |x-y| \geq \epsilon\}) \leq \epsilon\right\},
\end{equation}
where the infimum is taken over all the couplings $\pi$ of $\mu$ and $\nu$. We recall that a sequence of probability measures $\mu_n$ converges weakly to $\mu$ in $\Ps(\R^d)$ if and only if $\dP(\mu_n,\mu)$ converges to $0$, so that the metric topology associated with the Prohorov metric coincides with the topology of weak convergence~\cite[Theorem~6.8, p.~73]{Bil99}.

The following property of the Prohorov metric is immediate.

\begin{lem}[Translation invariance of the Prohorov metric]\label{lem:dPTI}
  For all $\mu, \nu \in \Ps(\R^d)$, for all $y \in \R^d$,
  \begin{equation*}
    \dP(\tau_y\mu,\tau_y\nu) = \dP(\mu,\nu).
  \end{equation*}
\end{lem}

For all $\bmu,\bnu \in \bPs(\R^d)$, let us define
\begin{equation}\label{eq:bdP}
  \bdP(\bmu,\bnu) := \inf\{\dP(\mu,\nu) : \rho(\mu)=\bmu, \rho(\nu)=\bnu\} \in [0,+\infty).
\end{equation}
For any $\mu, \nu \in \Ps(\R^d)$ such that $\rho(\mu)=\bmu$ and $\rho(\nu)=\bnu$, it is a consequence of Lemma~\ref{lem:dPTI} that $\bdP(\bmu,\bnu)$ rewrites
\begin{equation}\label{eq:bdP2}
  \bdP(\bmu,\bnu) = \inf\{\dP(\mu,\tau_y\nu) : y \in \R^d\}.
\end{equation}

\begin{lem}[Metrisability of $\bPs(\R^d)$]\label{lem:bdP}
  The function $\bdP$ is a metric on $\bPs(\R^d)$, and the associated metric topology is the same as the quotient topology.
\end{lem}
We call $\bdP$ the \emph{quotient Prohorov metric}.
\begin{proof}
  It is obvious that $\bdP$ is symmetric. To show that it satisfies the triangle inequality, we take $\bmu, \bnu, \blambda \in \bPs(\R^d)$ and fix $\mu, \nu, \lambda \in \Ps(\R^d)$ such that $\rho(\mu)=\bmu$, $\rho(\nu)=\bnu$ and $\rho(\lambda)=\blambda$. By~\eqref{eq:bdP} and the triangle inequality for $\dP$, for all $x,y \in \R^d$,
  \begin{equation*}
    \bdP(\bmu,\bnu) \leq \dP(\tau_x\mu, \tau_y\nu) \leq \dP(\tau_x\mu,\lambda) + \dP(\tau_y\nu,\lambda),
  \end{equation*}
  so that taking the infimum of the right-hand side of the inequality over $x,y \in \R^d$ and using~\eqref{eq:bdP2}, we obtain
  \begin{equation*}
    \bdP(\bmu,\bnu) \leq \bdP(\bmu,\blambda) + \bdP(\bnu,\blambda).
  \end{equation*}
  We now take $\bmu, \bnu \in \bPs(\R^d)$ such that $\bdP(\bmu,\bnu)=0$. Let $\mu, \nu \in \Ps(\R^d)$ such that $\rho(\mu)=\bmu$ and $\rho(\nu)=\bnu$. By~\eqref{eq:bdP2}, for all $n \geq 1$ there exists $y_n \in \R^d$ such that
  \begin{equation*}
    \dP(\mu,\tau_{y_n}\nu) \leq 1/n,
  \end{equation*}
  therefore $\tau_{y_n}\nu$ converges to $\mu$. By Ulam's Theorem~\cite[Theorem~1.3, p.~8]{Bil99}, $\nu$ is tight, hence there exists a centered ball $B(0,r)$, $r \geq 0$, such that
  \begin{equation}\label{eq:nuBr}
    \nu(B(0,r)) \geq 2/3.
  \end{equation}
  Likewise, by Prohorov's Theorem~\cite[Theorem~5.2, p.~60]{Bil99}, the family $(\tau_{y_n}\nu)_{n \geq 1}$ is tight, so that there exists $s \geq 0$ such that
  \begin{equation}\label{eq:nuBs}
    \forall n \geq 1, \qquad \tau_{y_n}\nu(B(0,s)) = \nu(B(y_n,s)) \geq 2/3.
  \end{equation}
  Assume that there exists an extracted sequence $(n_k)_{k \geq 1}$ such that $|y_{n_k}|$ diverges to $+\infty$: then for $k$ large enough, the balls $B(0,r)$ and $B(y_{n_k},s)$ are disjoint, so that the combination of~\eqref{eq:nuBr} and~\eqref{eq:nuBs} yields
  \begin{equation*}
    \nu\left(B(0,r) \cup B(y_{n_k},s)\right) \geq 4/3 > 1,
  \end{equation*}
  which is absurd. As a consequence, the sequence $(y_n)_{n \geq 1}$ is bounded and therefore possesses a converging subsequence, that we still index by $n$ for convenience, and the limit of which is denoted $y_*$. Using the continuity of the mapping $y \mapsto \tau_y\nu$, we get
  \begin{equation*}
    \mu = \lim_{n \to +\infty} \tau_{y_n}\nu = \tau_{y_*}\nu,
  \end{equation*}
  which implies that $\bmu = \rho(\mu) = \rho(\tau_{y_*}\nu) = \bnu$ and completes the proof that $\bdP$ is a metric.
  
  As an immediate consequence of the definition~\eqref{eq:bdP} of $\bdP$, we have the inequality
  \begin{equation*}
    \forall \mu,\nu \in \Ps(\R^d), \qquad \bdP(\rho(\mu), \rho(\nu)) \leq \dP(\mu,\nu),
  \end{equation*}
  which implies that $\rho$ is continuous for the metric topology induced on $\bPs(\R^d)$ by $\bdP$, so by definition of the quotient topology, the latter is stronger than the former. Now let $\bar{A}$ be an open set in the quotient topology. By the definition of the quotient topology, the set $A := \rho^{-1}(\bar{A})$ is open in $\Ps(\R^d)$, so that for all $\mu \in A$, there exists $r_\mu > 0$ such that $\mathcal{B}(\mu,r_\mu) \subset A$, whence
  \begin{equation*}
    A = \bigcup_{\mu \in A} \mathcal{B}(\mu,r_\mu).
  \end{equation*}
  Since $\rho(\tau_y\mu) = \rho(\mu) \in \bar{A}$ for any $y \in \R^d$ and $\mu \in A$, we may rewrite
  \begin{equation*}
    A = \bigcup_{y \in \R^d} \tau_y\left(\bigcup_{\mu \in A} \mathcal{B}(\mu,r_\mu)\right) = \bigcup_{\mu \in A}\bigcup_{y \in \R^d} \mathcal{B}(\tau_y\mu,r_\mu).
  \end{equation*}
  Introducing the notation 
  \begin{equation*}
    \bar{\mathcal{B}}(\bmu,r) := \{\bnu \in \bPs(\R^d) : \bdP(\bmu,\bnu) < r\},
  \end{equation*}
  we deduce from~\eqref{eq:bdP2} that, for all $\mu \in A$,
  \begin{equation*}
    \bigcup_{y \in \R^d} \mathcal{B}(\tau_y\mu,r_\mu) = \rho^{-1}\left(\bar{\mathcal{B}}(\rho(\mu),r_\mu)\right),
  \end{equation*}
  so that
  \begin{equation*}
    A = \bigcup_{\mu \in A} \rho^{-1}\left(\bar{\mathcal{B}}(\rho(\mu),r_\mu)\right) = \rho^{-1}\left(\bigcup_{\mu \in A} \bar{\mathcal{B}}(\rho(\mu),r_\mu)\right).
  \end{equation*}
  As a consequence,
  \begin{equation*}
    \bar{A} = \bigcup_{\mu \in A} \bar{\mathcal{B}}(\rho(\mu),r_\mu),
  \end{equation*}
  therefore $\bar{A}$ is an open set in the metric topology and the proof is completed. 
\end{proof}


\subsection*{Acknowledgements} This work was motivated by several discussions with Freddy Bouchet on the large deviations of mean-field particle systems. The author is grateful to Cyril Labbé for his careful reading of this manuscript, and thanks the referee for correcting a mistake in the proof of Lemma~\ref{lem:bdP}.


\end{document}